\documentclass[11pt]{amsart}
\usepackage{amsmath,amsthm,amssymb,amscd,cmbright}
\usepackage{amsrefs}
\usepackage[T1]{fontenc}
\usepackage[utf8]{inputenc}
\usepackage{lmodern}
\usepackage[pdftex]{hyperref}
\usepackage{amsfonts}
\usepackage{graphicx}
\usepackage{graphics}
\usepackage[english]{babel}
\usepackage{hhline}
\usepackage[dvips]{color}
\usepackage{pb-diagram}
\DeclareGraphicsExtensions{.png,.pdf,.jpg,.gif}

\newtheorem{Theorem}{Theorem}[section]

\newtheorem{lemma}[Theorem]{Lemma}
\newtheorem{Proposition}[Theorem]{Proposition}

\theoremstyle{definition}
\newtheorem{Remark}[Theorem]{Remark}
\newtheorem{example}[Theorem]{Example}

\newcommand{\R}{\mbox{${\mathbb R}$}}
\newcommand{\C}{\mbox{${\mathbb C}$}}

\newcommand{\Q}{\mbox{${\mathbb Q}$}}
\newcommand{\Z}{\mbox{${\mathbb Z}$}}

\newcommand{\w}{\mbox{${\omega}$}}

\newcommand{\cpi}{\mbox{${\mathbb {CP}^1}$}}

\begin{document}

\title[Upper bound for the Gromov width of coadjoint orbits]{\bf Upper bound for the Gromov width of coadjoint orbits of compact Lie
groups }
\author{Alexander Caviedes Castro}
\email{alexander.caviedescastro@mail.utoronto.ca}

\date{}

\maketitle

\begin{abstract}
We find an upper bound for the Gromov width of coadjoint orbits of
compact Lie groups with respect to the Kostant--Kirillov--Souriau
symplectic form by computing certain Gromov--Witten invariants. The
approach presented here is closely related to the one used by Gromov
in his celebrated non-squeezing theorem.
\end{abstract}

\section{Introduction}

The Gromov width of a symplectic manifold $(M^{2n}, \w)$ is defined
as
\[
\operatorname{Gwidth}(M^{2n}, \w):=\sup\{\pi r^2: \exists \text{ a
symplectic embedding } (B^{2n}(r),
\w_{\operatorname{st}})\hookrightarrow (M^{2n}, \w) \},
\]
where $B^{2n}(r):=\{(x_1, \cdots, x_n, y_1, \cdots,
y_n):\sum_{i=1}^n(x_i^2+y_i^2)<r\}$ denotes the open ball of radius
$r$ and center the origin in $\R^{2n}$ and
$\w_{\operatorname{st}}:=\sum_{i=1}^ndx_i\wedge dy_i$ denotes the
standard symplectic form defined on  $B^{2n}(r).$

The Darboux theorem in symplectic geometry implies that the Gromov
width of any symplectic manifold is always positive.

In this paper, we are interested in finding upper bounds for the
Gromov width of coadjoint orbits of compact Lie groups with respect
to the Kostant--Kirillov--Souriau form. The main result in this
paper is summarized in the following theorem:

\begin{Theorem}[Main Theorem] Let $G$ be a compact connected
simple Lie group with Lie algebra $\mathfrak{g}.$ Let $T\subset G$
be a maximal torus and let $\check{R}$ be the corresponding system
of coroots. We identify the dual Lie algebra $\mathfrak{t}^*$ with
the fixed points of the coadjoint action of $T$ on $\mathfrak{g}^*.$
For $\lambda \in \mathfrak{t}^* \subset \mathfrak{g}^*,$ let
$\mathcal{O}_\lambda$ be the coadjoint orbit passing through
$\lambda$  and $\w_\lambda$ be the Kostant--Kirillov--Souriau form
defined on $\mathcal{O}_\lambda.$ Then
\[
\operatorname{Gwidth}(\mathcal{O}_\lambda, \w_\lambda) \leq
\min_{\substack{\check{\alpha} \in \check{R} \\ \langle \lambda,
\check{\alpha} \rangle\ne 0}} |\langle \lambda, \check{\alpha}
\rangle|
\]
\end{Theorem}
The proof of the Main Theorem follows the same line of thought as
the one given by Gromov in his non-squeezing theorem \cite{gromov}.
It follows from Gromov's work that $J$-holomorphic curves can be
used to bound from above the Gromov width of a symplectic manifold.
Roughly speaking, if for a compact symplectic manifold $(M, \w)$
there is a non-vanishing Gromov-Witten invariant of the form
$\operatorname{GW}_{A,
k}(\operatorname{PD}[\operatorname{pt}],\alpha_2, \cdots, \alpha_k)$
for some degree $A\in H_2(M; \Z)$ and some cohomology classes
$\alpha_2, \cdots, \alpha_k\in H^*(M; \Z),$ then
$$
\operatorname{Gwidth}(M, \w) \leq \w(A)
$$
In this paper we use this approach to bound the Gromov width of
arbitrary coadjoint orbits of compact Lie groups. We adopt the
definition of the Gromov-Witten invariant given by Cieliebak and
Mohnke for integral symplectic manifolds \cite{cieliebak-mohnke} and
we use curve neighborhoods (Buch, Mihalcea and Perrin
\cite{BuchChaputMihalceaPerrin}) to compute them. We explain why in
our considerations these notions are consistent and imply the bound
appearing in the Main Theorem.

The Main Theorem extends a result of Zoghi \cite{zoghi} for
\textit{regular} coadjoint orbits of compact Lie groups that in
addition satisfy a condition that Zoghi named
\textit{indecomposable}. Recall that a coadjoint orbit of a compact
Lie group $G$ is regular if it is isomorphic to $G/T$ where $T$ is a
maximal torus of $G.$

In contrast to the problem of bounding the Gromov width of coadjoint
orbits from above, Pabiniak has considered the problem of bounding
it from below \cite{milena2}. For instance, Pabiniak has proved that
the upper bound appearing in the Main Theorem is indeed an equality
for arbitrary coadjoint orbits of $U(n)$ and certain coadjoint
orbits of $SO(n)$ with some technical assumptions made on $\lambda$.
Together with our result, this yields the following theorem for
coadjoint orbits of $U(n)$:

\begin{Theorem} Let us identify the Lie algebra of $U(n)$ with its dual via the
ad--invariant inner product
$$
(A, B) \to \operatorname{tr}(A\cdot B)
$$
For $(\lambda_1, \cdots, \lambda_n)\in \R^n,$ let
$\lambda=i\operatorname{diag}(\lambda_1, \cdots, \lambda_n) \in
\mathfrak{u}(n)\cong \mathfrak{u}(n)^*.$ Let $\mathcal{O}_\lambda$
be the coadjoint orbit of $U(n)$ passing through $\lambda\in
\mathfrak{u}(n)^*$ and $\w_\lambda$ be the Kostant-Kirillov-Souriau
form defined on the coadjoint orbit, then
$$
\operatorname{Gwidth}(\mathcal{O}_\lambda,
\w_\lambda)=\min_{\lambda_i \ne \lambda_j}|\lambda_{i}-\lambda_{j}|
$$
\end{Theorem}

A coadjoint orbit of $U(n)$ can be identified with a partial flag
manifold. Two special cases are the full flag manifold and the
Grassmannian manifold. The above theorem extends the results of G.
Lu \cite{glu}; Karshon and Tolman \cite{karshon} for Grassmannian
manifolds and Zoghi \cite{zoghi} for \textit{indecomposable} flag
manifolds to partial flag manifolds.

In addition to the examples metioned above, several authors have
used Gromov's method for bounding the Gromov width of other families
of symplectic manifolds such as G. Lu \cite{glu2}, McDuff and
Polterovich \cite{spacking} and Biran \cite{biran}.

This paper is organized as follows: in the second section, we review
the $J$-holomorphic tools that we use throughout the text, and then
we explain how upper bounds for the Gromov width of symplectic
manifolds manifolds can be obtained by a non-vanishing
Gromov--Witten invariant. In the third section, we recall background
on the geometry of coadjoint orbits of compact Lie groups and
homogeneous spaces. In the fourth section, we define the concept of
curve neighborhood and indicate its relation with Gromov-Witten
invariants. In the fifth section, we show the upper bound appearing
in the Main Theorem for Grassmannian manifolds. In the sixth
section, the Main Theorem is proven for non-regular coadjoint orbits
of compact Lie groups. In the appendix, we discuss about fibrations
of homogeneous spaces and show two technical lemmas that are needed
in the proof of the Main Theorem for Grassmannian manifolds.

\textbf{Acknowledgments} I would like to thank Yael Karshon for
introducing me this problem and for encouraging me during the
writing process of this paper. I also would like to thank to Milena
Pabiniak for useful conversations and helpful comments on the first
draft.

\section{$J$-holomorphic curves}

In this section we give a short review of pseudoholomorphic theory
and Gromov-Witten invariants, and we show how pseudoholomorphic
curves are related with the Gromov width of a symplectic manifold.
Most of the material presented here is adapted from McDuff and
Salamon \cite{mcduff}.

\subsection{Pseudoholomorphic curves}

Let $(M^{2n}, \w)$ be a compact symplectic manifold. An almost
complex structure $J$ on $(M, \w)$ is a smooth operator $J:TM\to TM$
such that $J^2=-\operatorname{Id}.$ We say that an almost complex
structure $J$ is \textbf{compatible} with $\w$ if
$$
g(\cdot\,, \cdot\,):=\w(\cdot\,, J\cdot\,)
$$
defines a Riemannian metric on $M$. We denote the space of almost
complex structures compatible with $\w$ by $\mathcal{J}(M, \w).$

Let $(\cpi, j)$ be the Riemann sphere with its standard complex
structure $j.$ Let $J\in \mathcal{J}(M, \w).$ A map $u:\cpi \to M$
is a \textbf{$J$-holomorphic curve of genus zero} or simply a
\textbf{$J$-holomorphic sphere}  if for every $z\in \cpi$
$$
du(z)+J(z)\circ du(z)\circ j=0
$$
The \textbf{degree} of a $J$-holomorphic sphere $u:\cpi \to M$ is
$$
\deg{u}:=u_{*}[\cpi]\in H_2(M; \Z)
$$
A homology class $A\in H_2(M; \Z)$ is \textbf{spherical} if it is in
the image of the Hurewicz homomorphism $\pi_2(M)\to H_2(M;\Z).$

A curve $u:\cpi \to M$ is said to be \textbf{multiply covered} if it
is the composite of a holomorphic branched covering map $(\cpi,
j)\to (\cpi, j)$ of degree greater than one with a $J$-holomorphic
map $\cpi \to M.$ It is \textbf{simple} if it is not multiply
covered.

Let $k$ be a nonnegative integer. A \textbf{nodal curve of genus
zero with $k$-marked points} is a tuple $(\Sigma; z_1, \cdots, z_k)$
consisting of a compact nodal Riemman surface $\Sigma$ of arithmetic
genus zero together with a set of $k$-pairwise different points
$z_1, \cdots, z_k$ that are not nodes of $\Sigma.$ We denote a nodal
curve with $k$-marked points $(\Sigma; z_1, \cdots, z_k)$ by
$\mathbf{z}$ and call $\Sigma_{\mathbf{z}}:=\Sigma$ the nodal
surface of $\mathbf{z}.$ A nodal curve with $k$-marked points
$\mathbf{z}$ is \textbf{stable} if the number of nodes and marked
points in any irreducible component of $\Sigma_{\mathbf{z}}$ is
greater or equal to three. Two nodal curves with  $k$-marked points
$\mathbf{z}=(\Sigma; z_1, \cdots, z_k), \mathbf{z}'=(\Sigma'; z_1',
\cdots, z_k')$ are equivalent if there exists an isomorphism of
curves $\phi:\Sigma \to \Sigma'$ such that $\phi$ maps $\{z_1,
\cdots, z_k\}$ bijectively onto $\{z_1', \cdots, z_k'\}.$ We denote
by $\overline{\mathcal{M}}_{k}$ the moduli space of  equivalent
classes of stable nodal curves of genus zero with $k$-marked points.

A \textbf{smooth curve with $k$-marked points} is a pair $(u\,;
\mathbf{z})$ consisting of a nodal curve with $k$-marked points
$\mathbf{z}$ and a smooth map $u:\Sigma_{\mathbf{z}} \to M.$ A
\textbf{smooth sphere with $k$-marked points} is a smooth curve with
$k$-marked points $(u\,; \mathbf{z})$ such that
$\Sigma_{\mathbf{z}}$ is isomorphic with $\cpi.$ A smooth curve with
$k$-marked points $(u\,; \mathbf{z})$ is $J$-holomophirc if the
restriction of $u$ on each of the irreducible components of
$\Sigma_{\mathbf{z}}$ is $J$-holomorphic. A $J$-holomorphic curve
$(u\,; \mathbf{z})$ is \textbf{simple} if $\Sigma_{\mathbf{z}}$ is
isomorphic with $\cpi$ and $u$ is simple. Two $J$-holomorphic curves
with $k$-marked points $(u;\, \mathbf{z}), (u';\, \mathbf{z}')$ are
\textbf{equivalent} if there exists an isomorphism
$\phi:\Sigma_{\mathbf{z}}\to \Sigma_{\mathbf{z'}}$ of curves such
that $u'=u\circ \phi$ and $\phi$ maps $\mathbf{z}$ bijectively onto
$\mathbf{z}'.$ A $J$-holomorphic curve with $k$-marked points $(u;
\mathbf{z})$ is \textbf{stable} if the number of marked and nodal
points in any component of $\Sigma_{\mathbf{z}}$ is greater or equal
to three whenever the restriction of $u$ on the component is
constant.

Let $A\in H_2(M; \Z).$ We denote by $\overline{\mathcal{M}}_{A,
k}(M, J)$ the moduli space of equivalent classes $[(u; \mathbf{z})]$
of stable $J$-holomorphic curves with $k$-marked points of degree
$\deg{\Sigma_{\mathbf{z}}}:=u_*[\Sigma_{\mathbf{z}}]=A.$ We denote
by $\mathcal{M}_{A, k}(M, J)$ the moduli space of equivalent classes
of $J$-holomorphic spheres $[(u; \mathbf{z})]$ in
$\overline{\mathcal{M}}_{A, k}(M, J)$ and by $\mathcal{M}_{A,
k}^*(M, J)$ the moduli space of equivalent classes of simple
$J$-holomorphic spheres $[(u; \mathbf{z})]$ in
$\overline{\mathcal{M}}_{A, k}(M, J).$ The moduli space
$\overline{\mathcal{M}}_{A, k}(M, J)$ is compact with respect to the
Gromov topology (see McDuff and Salamon \cite{mcduff}). The
\textbf{evaluation map}
$$
\operatorname{ev}_J^k=(\operatorname{ev}_1, \ldots,
\operatorname{ev}_k):\overline{\mathcal{M}}_{A, k}(M, J) \to M^k
$$
is defined by
$$
\operatorname{ev}_J^k[u;\,(\Sigma;\, z_1, \cdots, z_k)]=(u(z_1),
\cdots, u(z_k)).
$$

We denote by $\mathcal{J}_{\operatorname{reg}}(M, \w)\subset
\mathcal{J}(M, \w)$ the set of almost complex structures that are
regular in the sense of McDuff and Salamon \cite[Definition 3.1.4,
Section 6.2]{mcduff}. When $J$ is a regular almost complex
structure, the moduli space $\mathcal{M}_{A, k}^*(M, J)$ is a smooth
oriented manifold of dimension equal to
$$
\dim{M}+2c_1(A)+2k-6,
$$
where $c_1$ denotes the first Chern class of $TM$ (see e.g. McDuff
and Salamon \cite{mcduff}).

\subsection{The Gromov width and pseudoholomorphic curves}
For $r>0,$ let
$$
B^{2n}(r):=\Bigl\{(x_1, \cdots, x_n, y_1, \cdots, y_n)\in
\R^{2n}:\sum_{i=1}^n(x_i^2+y_i^2)<r^2\Bigr\}
$$
denote the open ball in $\R^{2n}$ of radius $r$ and center the
origin. Let
$$
\w_{\operatorname{st}}:=\sum_{i=1}^n dx_i\wedge dy_i
$$ be the standard symplectic
form defined on $B^{2n}(r).$ Given a symplectic manifold $(M^{2n},
\w),$ its \textbf{Gromov's width} is defined as
$$
\operatorname{Gwidth}(M^{2n}, \w):=\sup{\{\pi r^2: \exists\text{ a
symplectic embedding }  (B^{2n}(r), \w_{\operatorname{st}})
\hookrightarrow (M^{2n}, \w)\}}
$$
The Darboux theorem implies that the Gromov width of a symplectic
manifold is always positive. Moreover, if the symplectic manifold is
compact, its Gromov width is finite.

For a symplectic manifold $(M, \w)$ and a homology class $A\in
H_2(M; \Z)$ we denote by $\w(A)$ the symplectic area of $A.$

The following statement goes back to Gromov and shows the relation
between the Gromov width of a symplectic manifolds and its
pseudoholomorphic curves.

\begin{Theorem}[Gromov \cite{gromov}]\label{nsq}
Let $(M^{2n}, \w)$ be a compact symplectic manifold and $A\in H_2(M;
\Z)$ be a nontrivial spherical class. Suppose that for any almost
complex structure $\tilde{J}\in \mathcal{J}(M, \w)$ and for any
point $p\in M$ there exists a $\tilde{J}$-holomorphic sphere of
degree $B$ passing through $p$ with $0<\w(B)\leq \w(A).$ Then for
any symplectic embedding $B^{2n}(r)\hookrightarrow M,$ we have
$$
\pi r^2 \leq \w(A)
$$
In particular,
$$
\operatorname{Gwidth}(M, \w)\leq \w(A).
$$
\end{Theorem}
\begin{proof}
Let us suppose that there exists a symplectic embedding
$$
\rho: B^{2n}(r)\hookrightarrow M
$$
Let $J_{\operatorname{st}}$ be the standard almost complex structure
defined on the open ball $B^{2n}(r).$ Given $\epsilon\in (0, r),$
let $\tilde{J}$ be an almost complex structure that is equal to
$\rho_*(J_{\operatorname{st}})$ on the open set
$D:=\rho(B^{2n}(r-\epsilon))\subset M.$ By assumption, there exists
a $\tilde{J}$-holomorphic sphere $u:\cpi \to M$ of degree $B$
passing through $\rho(0)$ with $0<\w(B)\leq \w(A).$ The restriction
of $u$ to $S:=u^{-1}(D)\subset \cpi,$ gives a proper holomorphic map
$$
u':S\to B^{2n}(r-\epsilon)
$$
that passes through the origin. The monotonicity theorem implies
that the area of this curve is bounded from below by
$\pi(r-\epsilon)^2.$ As a consequence,
$$
\pi(r-\epsilon)^2 \leq  \operatorname{area}(u')\leq
\operatorname{area}(u)=\w(B) \leq \w(A)
$$
Since this inequality is true for every $\epsilon\in (0, r),$
$$
\pi r^2 \leq \w(A) \qedhere
$$
\end{proof}

\subsection{Pseudocycles}

In order to define the Gromov-Witten invariants, we review the
notion of pseudocycle. A subset $\Omega\subset X$ of a manifold $X$
has \textbf{dimension at most $d$} if $\Omega$ is contained in the
image of a smooth map $g:N\to X$ from a manifold of dimension
$\dim{N}\leq d.$ A smooth map $f:M\to X$ from an oriented
$d$-manifold $M$ to a manifold $X$ is called \textbf{$d$-dimensional
pseudocycle} if $f(M)$ has compact closure and its omega limit set
$$
\Omega_f:=\bigcap_{K\subset M \text{ compact}}\overline{f(M-K)}
$$
has dimension at most $d-2.$ Roughly speaking this means that $f(M)$
can be compactified by adding strata of codimension at least two.

Two $d$-dimensional pseudocycles $f_i:M_i \to X, i=0, 1,$ are called
\textbf{cobordant} if there exists a smooth map $F:W\to X$ from an
oriented $(d+1)$-manifold $W$ with boundary $\partial W=M_1-M_0$
such that $F|_{M_i}=f_i$ and $\dim \Omega_F\leq d-1.$

Two pseudocycles $f_i:M_i\to X, i=0, 1, $ are called
\textbf{strongly transverse} if
$$
\Omega_{f_0}\cap \overline{f_{1}(M_{1})}=\emptyset, \,\,\,\,
\overline{f_{0}(M_{0})}\cap \Omega_{f_1} =\emptyset
$$
and $f_1(m_1)=f_2(m_2)=x$ implies
$T_xX=\operatorname{im}df_1|_{m_1}+\operatorname{im}df_2|_{m_2}.$

There exists a set $\operatorname{Diff}_{\operatorname{reg}}(X, f_1,
f_2)\subset \operatorname{Diff}(X)$ of the second category such that
$f_1$ is strongly transverse to $\phi\circ f_2$ for every $\phi\in
\operatorname{Diff}_{\operatorname{reg}}(X, f_1, f_2).$ If $f_1$ and
$f_2$ are strongly transverse, the set $\{(m_1, m_2)\in M_1\times
M_2:f_1(m_1)=f_2(m_2)\}$ is a compact manifold of dimension
$\dim{M_1}+\dim{M_2}-\dim{X}.$ In particular, this set is finite if
$M_1$ and $M_2$ are of complementary dimension. In this case, the
\textbf{intersection number} of $f_1$ and $f_2$ is defined as
$$
\sharp( f_1\pitchfork f_2):=\sum_{\substack{m_1\in M_1, m_2\in M_2 \\
f_1(m_1)=f_2(m_2)}} f_1(m_1)\cdot f_2(m_2),
$$
where $f_1(m_1)\cdot f_2(m_2)$ denotes the intersection number of
$f_1(M_1)$ and $f_2(M_2)$ at $f_1(m_1)=f_2(m_2).$ The intersection
number $f_1\pitchfork f_2$ depends only on the bordism classes of
$f_1$ and $f_2$ (McDuff and Salamon \cite{mcduff}[Section 6.5]).

\subsection{The Gromov width and Gromov-Witten
invariants of semipositive symplectic
manifolds}\label{gwidthgwinvariants}

A symplectic manifold $(M^{2n}, \w)$ is \textbf{semipositive} if,
for a spherical class $A$ with positive symplectic area, $c_1(A)\geq
3-n$ implies $c_1(A)\geq 0.$

Let $(M, \w)$ be a semipositive symplectic manifold. Let $A$ be a
spherical class and  $k$ be a nonnegative integer. For $J\in
\mathcal{J}_{\operatorname{reg}}(M, \w),$ the evaluation map
$$
\operatorname{ev}_J^k:\mathcal{M}^*_{A, k}(M, J) \to M^k
$$
defines a pseudocycle of dimension equal to
$$
\dim{M}+ 2c_1(A)+2k-6
$$
The pseudocycle $\operatorname{ev}_J^k:\mathcal{M}^*_{A, k}(M, J)
\to M^k,$ up to cobordism,  is independent of the regular almost
complex structure $J\in \mathcal{J}_{\text{reg}}(M, \w)$ (see e.g.
McDuff and Salamon \cite[Chapter 6]{mcduff}).

Let $J\in \mathcal{J}_{\operatorname{reg}}(M, \w)$ and $a_1, \cdots,
a_k$ be cohomology classes of total degree
$$
\sum_{i=1}^k \deg{a_i}=\dim{M}+2c_1(A)+2k-6
$$
If the co\-ho\-mo\-lo\-gy classes $a_i$ are Poincar\'e dual to the
fundamental classes of cycles $X_i\subset M$ so the evaluation map
$\operatorname{ev}_J^k:\mathcal{M}^*_{A, k}(M, J) \to M^k $ is
transversal to $X_1\times \cdots \times X_k \subset M^k,$ the
\textbf{Gromov-Witten invariant} $\operatorname{GW}^J_{A, k}(a_1,
\ldots,a_k)$ is the intersection number
$$
\operatorname{GW}_{A, k}^J(a_1,\ldots, a_k):=\sharp
\operatorname{ev}_J^{k}\pitchfork(X_1\times \cdots\times X_k)
$$
The Gromov-Witten invariant $\operatorname{GW}_{A, k}^J(a_1,\ldots,
a_k)$ is well-defined, finite and independent of the regular almost
complex structure $J$ and of generic perturbation of $X_1\times
\cdots \times X_k$ (McDuff and Salamon \cite[Theorem 6.6.1, Theorem
7.1.1]{mcduff}).

\begin{Remark}\label{gwj}
Let $(M, \w)$ be a compact symplectic manifold that is not
necessarily semipositive. For $J\in \mathcal{J}(M, \w),$ we say that
a homology class $B\in H_2(M; \Z)$ is $J$-\textbf{indecomposable} if
it can not be decomposed as a sum $B=B_1+\cdots +B_l,$ where $l\geq
2$ and each $B_i$ has a nonconstant spherical $J$-holomorphic
representative.

Let $A\in H_2(M; \Z)$ and $k\in \Z_{\geq 0}.$ There exists a subset
$\mathcal{J}_{\operatorname{reg}}(A) \subset \mathcal{J}(M, \w)$ of
the second category such that for any $J\in
\mathcal{J}_{\operatorname{reg}}(A)$ the moduli space
$\mathcal{M}^*_{A, k}(M, J)$ is a smooth manifold. If $J\in
\mathcal{J}_{\operatorname{reg}}(A)$ and $A$ is $J$-indecomposable,
the evaluation map
$$
\operatorname{ev}_J^k:\mathcal{M}^*_{A, k}(M, J) \to M^k
$$
defines a pseudocycle and the Gromov-Witten invariant
$\operatorname{GW}_{A, k}^J$ can be defined using
$\operatorname{ev}_J^k.$ This Gromov-Witten invariant is defined
exclusively for the regular almost complex structure $J$ and the
$J$-indecomposable class $A$ (see e.g. McDuff and Salamon
\cite{mcduff}[Lemma 7.1.8]).
\end{Remark}

The following theorem allows us to bound the Gromov width of a
semipositive symplectic compact manifold in terms of its
Gromov-Witten invariants.

\begin{Theorem}\label{nsq2}
Let $(M, \w)$ be a compact semipositive symplectic manifold. If
there exist a regular almost complex structure $J,$ a spherical
class $A$ and cohomology classes $a_2, \ldots, a_k$ Poincar\'e dual
to the fundamental classes of cycles $X_2, \ldots, X_k \subset M$
such that
$$
\operatorname{GW}_{A, k}^J(\operatorname{PD}[\operatorname{pt}],
a_2, \ldots, a_k)\ne 0,
$$
then
$$
\operatorname{Gwidth}(M, \w)\leq \w(A)
$$
\end{Theorem}
\begin{proof}
For any $\tilde{J}\in \mathcal{J}_{\operatorname{reg}}(M, \w),$ the
evaluation map
$$
\operatorname{ev}_{\tilde{J}}^k:\mathcal{M}^*_{A, k}(M,
\tilde{J})\to M^k
$$
defines a pseudocycle. For a generic perturbation
$\tilde{\operatorname{pt}}\times \tilde{X}_2\times \cdots \times
\tilde{X}_k$ of $\operatorname{pt}\times X_2\times \cdots \times
X_k$ such that the evaluation map $\operatorname{ev}_{\tilde{J}}^k$
is transversal to $\tilde{\operatorname{pt}}\times \tilde{X}_2\times
\cdots \times \tilde{X}_k,$ the Gromov-Witten invariant
$\operatorname{GW}_{A,k}^{\tilde{J}}(\operatorname{PD}[\operatorname{pt}],a_2,
\cdots, a_k)$ counts with appropriate orientations the number of
simple $\tilde{J}$-holomorphic spheres of degree $A$ passing through
$\tilde{\operatorname{pt}}, \tilde{X}_2 , \cdots, \tilde{X}_k.$ By
assumption, this number is different from zero and in particular for
a generic point in $M$ there exists a $\tilde{J}$-holomorphic sphere
of degree $A$ passing through it.

Let $\tilde{J}\in \mathcal{J}(M, \w).$ The set of regular almost
complex structure $\mathcal{J}_{\operatorname{reg}}(M, \w)$ is dense
in $\mathcal{J}(M, \w)$ with its $C^{\infty}$-topology. Thus, there
is a sequence of regular almost complex structures $\tilde{J}_n\in
\mathcal{J}(M, \w)$ that $C^\infty$-converges to $\tilde{J}.$ Also,
for any point $p\in M,$ we can find a sequence of points $p_n\in M$
that converges to $p$ and a sequence of $\tilde{J}_n$-holomorphic
spheres
$$
u_n:\cpi \to M
$$
of degree $A$ such that $u_n$ passes through $p_n.$ The Gromov
compactness theorem implies that there exists a
$\tilde{J}$-holomorphic sphere $u:\cpi \to M$ of degree $B$ passing
through $p$ with $0<\w(B)\leq \w(A).$  Therefore, by Theorem
\ref{nsq}
$$
\operatorname{Gwidth}(M, \w)\leq \w(A) \qedhere
$$
\end{proof}

\subsection{The Gromov width and Gromov-Witten
invariants of integral symplectic manifolds}\label{gwcm}

In this section we briefly review the definition of Gromov-Witten
invariants provided by Cieliebak and Monhke for integral symplectic
manifold in terms of Donaldson hypersurfaces. Most of the material
presented in this section is adapted from Cieliebak and Monhke
\cite{cieliebak-mohnke}.

Let $l$ be a nonnegative integer and $\overline{\mathcal{M}}_{l}$ be
the moduli space of equivalent classes of stable nodal curves of
genus zero with $l$-marked points. Let $\operatorname{st}$ be the
\textbf{stabilization map} that makes any nodal curve with
$l$-marked points $\mathbf{z}$ into a stable nodal curve
$\operatorname{st}(\mathbf{z}).$ The stabilization map determines a
holomorphic surjection on the corresponding nodal curves
$$
\operatorname{st}: \Sigma_{\mathbf{z}} \to
\Sigma_{\operatorname{st}{(\mathbf{z})}}
$$
We denote by
$$
\pi:\overline{\mathcal{M}}_{l+1} \to \overline{\mathcal{M}}_{l}
$$
the map that forgets the last marked point and stabilizes the
result.

Let $(M, \w)$ be a symplectic manifold. A \textbf{domain-dependent
almost complex structure} is a smooth function
\begin{align*}
K:\overline{\mathcal{M}}_{l+1} &\to
\mathcal{J}(M, \w) \\
[\mathbf{z}] &\mapsto K_{[\mathbf{z}]}
\end{align*}
A \textbf{domain-independent almost complex structure} is an almost
complex structure $J\in \mathcal{J}(M, \w)$ considered as a constant
function
\begin{align*}
J:\overline{\mathcal{M}}_{l+1} &\to
\mathcal{J}(M, \w) \\
[\mathbf{z}] &\mapsto J
\end{align*}
We interpret $K\in C^{\infty}(\overline{\mathcal{M}}_{l+1},
\mathcal{J}(M, \w))$ in terms of the forgetful and stabilization
maps as follows: for any nodal curve $\mathbf{z}$ with $l$-marked
points, the fiber $\pi^{-1}([\operatorname{st}(\mathbf{z})])$ can be
identified with the nodal curve
$\Sigma_{\operatorname{st}(\mathbf{z})}$ by parameterizing
$\pi^{-1}([\operatorname{st}(\mathbf{z})])$  via the position of the
extra marked point. The restriction of  $K$ to
$\pi^{-1}([\operatorname{st}(\mathbf{z})])$ yields a map
$$
K_{[\mathbf{z}]}^{\operatorname{st}}:\Sigma_{\operatorname{st}(\mathbf{z})}\to
\mathcal{J}(M, \w)
$$
We denote by $\mathcal{J}_{l+1}(M, \w)$ the set of domain-dependent
almost complex structures that are \textbf{coherent} in the sense of
Cielibak and Mohnke \cite{cieliebak-mohnke}.

Let $K\in \mathcal{J}_{l+1}(M, \w).$ A smooth curve $(u,
\mathbf{z})$ with $l$-marked points is \textbf{$K$-holomorphic} if
for every $z\in \Sigma_{\mathbf{z}},$
$$
du(z)+K_{[\mathbf{z}]}^{\operatorname{st}}(\operatorname{st}(z),
u(z))\circ du(z)\circ j_{\operatorname{st}(\mathbf{z})}=0,
$$
where $j_{\operatorname{st}(\mathbf{z})}$ denotes the standard
complex structure defined on
$\Sigma_{\operatorname{st}(\mathbf{z})}.$ For $A\in H_2(M; \Z),$ we
denote by $\mathcal{M}_{A, l}(M, K)$ the moduli space of equivalence
classes of $K$-holomorphic spheres of degree $A$ with $l$-marked
points and by $\operatorname{ev}_K^l:\mathcal{M}_{A, l}(M, K)\to
M^l$ the corresponding evaluation map.

Let us assume now that  $(M, \w)$ is a symplectic manifold such that
the symplectic form $\w$ represents a cohomology class $[\w]$ with
integer coefficients. We call such symplectic manifolds
\textbf{integral}. A codimension two submanifold $Y\subset M$
(\textbf{a hypersurface}) has \textbf{degree $D$} if it represents
the homology class Poincar\'e dual to $D[\w].$

A result of Donaldson states that given an almost complex structure
$J\in \mathcal{J}(M, \w)$ there exists a hypersurface $Y\subset M$
\textit{nearly} $J$-holomorphic in the sense that the
\textit{K\"{a}hler angle} of $Y$ with respect to $J$ is sufficiently
small if the degree of $Y$ is sufficiently large. This in particular
implies that if the degree of $Y$ is sufficiently large, for any
$\epsilon >0$ one can find $\tilde{J}\in \mathcal{J}(M, \w)$ such
that $\|J-\tilde{J}\|_{C^0}<\epsilon$ and $Y$ is
$\tilde{J}$-holomorphic (see e.g. Donaldson \cite{donaldson}). In
this case, we call $Y$ a \textbf{Donaldson hypersurface} and $(J,
Y)$ a \textbf{Donaldson pair}.

Let $(J, Y)$ be a Donaldson pair and $D$ be the degree of $Y.$ For
$\epsilon>0$ and $l\in \Z_{\geq 0},$ we denote by
$\mathcal{J}_{l+1}(M, \w, Y, \epsilon)$ the set of domain-dependent
almost complex structure $K\in \mathcal{J}_{l+1}(M, \w)$ that are
$\epsilon$-close to $J$ in the $C^0$-norm and leave $TY$ invariant.

Let $k\in \Z_{\geq 0}, A\in H_2(M; \Z), l:=D\w(A)$ and $K\in
\mathcal{J}_{l+1}(M, \w, Y, \epsilon).$ The domain-dependent almost
complex structure $K$ induces an element $\pi_k^*K\in
\mathcal{J}_{k+l+1}(M, \w)$ by composition with the map
$$
\pi_k:\overline{\mathcal{M}}_{k+l+1}\to \overline{\mathcal{M}}_{l+1}
$$
that forgets the last $k$ marked points and stabilizes the result.
We denote by $\mathcal{M}_{A, k+l}(M, K, Y)$ the moduli space of
equivalence classes of $\pi_k^*K$-holomorphic spheres of degree $A$
with $(k+l)$-marked points that map the last $l$ marked points to
$Y.$ We denote by $\operatorname{ev}^k_K:\mathcal{M}_{A, k+l}(M, K,
Y)\to M^k$ the evaluation map $ [u; (\Sigma; z_1, \cdots, z_k,
z_{k+1}, \cdots, z_{k+l})]\mapsto (u(z_1), \cdots, u(z_k)).$

\begin{Theorem}[Cieliebak and Mohnke
\cite{cieliebak-mohnke}]\label{cielibakmohnke} Let $(M, \w)$ be an
integral symplectic manifold and  $(J, Y)$ be a Donaldson pair
consisting of an almost complex structure $J\in \mathcal{J}(M, \w)$
and a Donaldson hypersurface $Y\subset M$ of degree $D$ sufficienlty
large. Let $A\in H_2(M; \Z)$ and $l=D\w(A).$ For $\epsilon>0$ small
enough there exists a subset
$\mathcal{J}^{\operatorname{reg}}_{l+1}(M, \w, Y, \epsilon)\subset
\mathcal{J}_{l+1}(M, \w, Y, \epsilon) $ of the second category such
that for every $k\in \Z_{\geq 0}$ and $K\in
\mathcal{J}^{\operatorname{reg}}_{l+1}(M, \w, Y, \epsilon)$ the
evaluation map
\begin{align*}
\operatorname{ev}^k_K:\mathcal{M}_{A, k+l}(M, K, Y)&\to M^k
\end{align*}
defines a pseudocycle. The bordism class of this pseudocycle is
independent of $(J, K, Y).$
\end{Theorem}

The previous theorem allows us to define for an arbitrary integral
symplectic manifolds its Gromov-Witten invariants. Let $(J, K, Y)$
and
$$
\operatorname{ev}^k_K:\mathcal{M}_{A, k+l}(M, K, Y)\to M^k
$$
be as in the previous theorem. Let $a_1, \cdots, a_k$ be cohomology
classes Poincar\'e dual to the fundamental classes of cycles $X_1,
\cdots, X_k\subset M$ such that the pseudocycle
$\operatorname{ev}^k_K:\mathcal{M}_{A, k+l}(M, K, Y)\to M^k$ is
transversal to $X_1\times \cdots \times X_k.$ \textbf{The
Gromov-Witten invariant defined by Cieliebak and Monhke} is the
intersection number
\begin{align*}
\operatorname{GW}^{\,\operatorname{cm}}_{A, k}(a_1,\ldots, a_k):=
\dfrac{1}{l!}\sharp\operatorname{ev}^k_K\pitchfork(X_1\times
\cdots\times X_k)
\end{align*}
The Gromov-Witten invariant
$\operatorname{GW}^{\,\operatorname{cm}}_{A, k}$ is well defined. It
does not depend on the choice of $(J, Y, K)$ and of generic
perturbation of $X_1, \cdots, X_k\subset M.$

The Gromov-Witten defined by Cieliebak and Mohnke coincides with the
Gromov-Witten invariant defined in the previous section for
(integral) semipositive symplectic manifolds
\cite{cieliebak-mohnke}.

Now we show that for an integral symplectic manifold $(M, \w),$ the
Gromov-Witten invariant $\operatorname{GW}_{A,
k}^{\operatorname{cm}}$ coincides with the Gromov-Witten invariant
$\operatorname{GW}_{A, k}^{J}$ when $J$ is a regular
domain-independent almost complex structure and $A$ is a
$J$-indecomposable spherical class.

\begin{lemma} Let $(M, \w)$ be a compact integral symplectic
manifold. Let $A\in H_2(M; \Z)$ and $J\in
\mathcal{J}_{\operatorname{reg}}(A).$ Assume that $A$ is a
$J$-indecomposable class. Then for any $k\in \Z_{\geq 1},$ the
Gromov-Witten invariant $\operatorname{GW}_{A, k}^{J}$ coincides
with the Gromov-Witten invariant $\operatorname{GW}_{A,
k}^{\operatorname{cm}}$
\end{lemma}
\begin{proof}
Let $\tilde{k}\in \Z_{\geq 1}.$ There exists a subset
$\mathcal{J}_{\tilde{k}+1}^{\operatorname{reg}}(A)\subset
\mathcal{J}_{\tilde{k}+1}(M, \w)$ of the second category such that
for every $K\in \mathcal{J}_{\tilde{k}+1}^{\operatorname{reg}}(A),$
the moduli space of $K$-holomorphic spheres
$\mathcal{M}_{A,\tilde{k}}(M, K)$ is a smooth manifold of dimension
$$
\dim{M}+2c_1(A)+2\tilde{k}-6,
$$
(see e.g. Cieliebak and Mohnke \cite{cieliebak-mohnke}[Corollary
5.8]). A standard transversality argument implies that if $K_0, K_1
\in \mathcal{J}_{\tilde{k}+1}^{\operatorname{reg}}(A)$ can be joined
by a path of domain-dependent almost complex structures, we can find
a path $\{K_t\}_{t\in [0, 1]} \subset \mathcal{J}_{\tilde{k}+1}(M,
\w)$ that joins $K_0$ with $K_1$ such that the moduli space
$$
\mathcal{W}_{A, \tilde{k}}(M, \{K_t\}):=\bigcup_{t\in [0,
1]}\mathcal{M}_{A,\tilde{k}}(M, K_t)
$$
is a smooth oriented manifold with boundary
$$
\partial(\mathcal{W}_{A, \tilde{k}}(M, \{K_t\}))=\mathcal{M}_{A, \tilde{k}}(M,
K_1)-\mathcal{M}_{A, \tilde{k}}(M, K_0)
$$
We claim that if $K\in
\mathcal{J}_{\tilde{k}+1}^{\operatorname{reg}}(A)$ is close enough
to $J$ and can be joined to $J$ by a path of domain-dependent almost
complex structures close enough to $J,$ the evaluation maps
$$
\operatorname{ev}_{K}^{\tilde{k}}:\mathcal{M}_{A, \tilde{k}}(M, K)
\to M^{\tilde{k}},\,\,
\operatorname{ev}_{J}^{\tilde{k}}:\mathcal{M}_{A, \tilde{k}}(M, J)
\to M^{\tilde{k}}
$$
define cobordant pseudocycles: first, for any almost complex
structure $K\in \mathcal{J}_{\tilde{k}+1}(M, \w),$ we say that a
homology class $B\in H_2(M)$ is $K$-indecomposable if it can not be
decomposed as a sum $B=B_1+\cdots +B_l,$ where $l\geq 2$ and each
$B_i$ has a nonconstant spherical $K$-holomorphic representative.

The set
$$
\mathcal{J}_{A, \tilde{k}+1}^{\operatorname{ind}}:=\{K\in
\mathcal{J}_{\tilde{k}+1}(M, \w): A \text{ is $K$-indecomposable}\}
$$
is an open subset of $\mathcal{J}_{\tilde{k}+1}(M, \w).$ If $K\in
\mathcal{J}_{\tilde{k}+1}^{\operatorname{reg}}(A)$ is close enough
to $J$ and satisfies the assumption, we can assume that $A$ is
$K$-indecomposable and that we can find a path $\{K_t\}_{t\in [0,
1]} \subset \mathcal{J}_{A, \tilde{k}+1}^{\operatorname{ind}}$ that
joins $K$ with $J$ such that $\mathcal{W}_{A, \tilde{k}}(M,
\{K_t\})$ is a smooth oriented manifold with boundary
$\partial{\mathcal{W}_{A, \tilde{k}}(M, \{K_t\})}=\mathcal{M}_{A,
\tilde{k}}(M, K)-\mathcal{M}_{A, \tilde{k}}(M, J).$

The indecomposability assumption made on $A$ implies that the
evaluation maps
$$
\operatorname{ev}_{K}^{\tilde{k}}:\mathcal{M}_{A, \tilde{k}}(M, K)
\to M^{\tilde{k}},\,\,
\operatorname{ev}_{J}^{\tilde{k}}:\mathcal{M}_{A, \tilde{k}}(M, J)
\to M^{\tilde{k}}
$$
are pseudocycles and that the evaluation map
$$
\operatorname{ev}_{K_t}^{\tilde{k}}:\mathcal{W}_{A, \tilde{k}}(M,
\{K_t\}) \to M^{\tilde{k}}
$$
defines a cobordism between the pseudocycles
$\operatorname{ev}_{K}^{\tilde{k}}$ and
$\operatorname{ev}_{J}^{\tilde{k}}$. This proves the claim.

Now, let $Y$ be a Donaldson hypersurface of degree $D$ sufficiently
large such that $(J, Y)$ is a Donaldson pair. Let $k\in \Z_{\geq 0}$
and $l=D\w(A).$ For $\epsilon>0$ small enough, let $K\in
\mathcal{J}_{l+1}^{\operatorname{reg}}(M, \w, Y, \epsilon)\subset
\mathcal{J}_{l+1}^{\operatorname{reg}}(A)$ be such that the
evaluation map
$$
\operatorname{ev}^k_K:\mathcal{M}_{A,k+l}(M, K, Y)\to M^k
$$
defines a pseudocycle.

If $K$ is close enough to $J,$ we can join $K$ with $J$ by a path of
domain-dependent almost complex structures close enough to $J$ (see
e.g. Cieliebak and Monhke \cite{cieliebak-mohnke}[Definition 9.4,
Proposition 10.1]). Thus the evaluation maps
$$
\operatorname{ev}^{k+l}_{\pi_l^*K}: \mathcal{M}_{A,k+l}(M,
\pi_l^*K)\to M^{k+l}, \,\,\, \operatorname{ev}^{k+l}_{J}:
\mathcal{M}_{A,k+l}(M, J)\to M^{k+l}
$$
are cobordant pseudocycles. As a consequence, for cycles in general
position $X_1, \cdots, X_k\subset M$ whose fundamental classes are
Poincar\'e dual to cohomology classes $a_1, \cdots, a_k\in H^*(M;
\Z)$
\begin{align*}
\operatorname{GW}^{\operatorname{cm}}_{A, k}(a_1, \cdots, a_k)&:=
\dfrac{1}{l!}\sharp\operatorname{ev}_{\pi_l^*K}^{k+l}\pitchfork (X_1
\times \cdots
\times X_k\times Y \times \cdots \times Y)\\
&=\dfrac{1}{l!}\sharp\operatorname{ev}_{J}^{k+l}\pitchfork (X_1
\times \cdots \times X_k\times Y \times \cdots \times Y)
\end{align*}
Finally, the divisor axiom implies that
\begin{align*}
\operatorname{GW}^{J}_{A, k}(a_1, \cdots, a_k)&:=
\sharp\operatorname{ev}_{J}^{k}\pitchfork (X_1 \times \cdots \times
X_k)\\&=\dfrac{1}{l!}\sharp\operatorname{ev}_{J}^{k+l}\pitchfork
(X_1 \times \cdots \times X_k\times Y \times \cdots \times Y),
\end{align*}
and we are done.
\end{proof}
For our purposes, the following two results would be enough to bound
the Gromov width of a symplectic manifold.
\begin{lemma}\label{nsq3}
Let $(M, \w)$ be a compact integral symplectic manifold. If there
exist a spherical class $A$ and cohomology classes $a_2, \cdots,
a_k$ such that
$$
\operatorname{GW}^{\operatorname{cm}}_{A,
k}(\operatorname{PD}[\operatorname{pt}], a_2,\ldots, a_k)\ne 0,
$$
then for any almost complex structure $\tilde{J}\in \mathcal{J}(M,
\w)$ and for any point $p\in M$ there exists a
$\tilde{J}$-holomorphic sphere of degree $B$ passing through $p$
with $0<\w(B)\leq \w(A)$ .
\end{lemma}
\begin{proof}
Assume that there exist a spherical class $A\in H_2(M; \Z)$ and
cohomology classes $a_2, \cdots, a_k$ Poincar\'e dual to the
fundamental classes of cycles $X_2, \cdots, X_k$ in general position
such that
$$
\operatorname{GW}^{\operatorname{cm}}_{A,
k}(\operatorname{PD}[\operatorname{pt}], a_2,\ldots, a_k)\ne 0
$$
The definition of the Gromov-Witten invariant
$\operatorname{GW}^{\operatorname{cm}}_{A, k}$ implies that for any
almost complex structure $\tilde{J}\in \mathcal{J}(M, \w)$ we can
construct a domain-dependent almost complex structure $K$
sufficiently close to $\tilde{J}$ such that for a generic point
$p\in M$ there exists a $K$-holomorphic sphere
$$
u:\cpi \to M
$$
of degree $A$ passing through $p.$ The Gromov compactness theorem
implies that for any point $p\in M$ there exists a
$\tilde{J}$-holomorphic curve of degree $B$ passing through $p$ with
$0<\w(B)\leq \w(A).$
\end{proof}
\begin{Theorem}\label{nsq4}
Let $(M, \w)$ be a compact symplectic manifold such that the
symplectic form $\w$ represents a rational cohomology class $[\w]\in
H^2(M; \Q).$ Let $A\in H_2(M; \Z)$ and $J\in
\mathcal{J}_{\operatorname{reg}}(A).$ Assume that $A$ is a
$J$-indecomposable class. If there exist cohomology classes $a_2,
\cdots, a_k\in H^*(M; \Z)$ such that
$$
\operatorname{GW}^{J}_{A, k}(\operatorname{PD}[\operatorname{pt}],
a_2,\ldots, a_k)\ne 0,
$$
then for any almost complex structure $\tilde{J}\in \mathcal{J}(M,
\w)$ and for any point $p\in M$ there exists a
$\tilde{J}$-holomorphic sphere of degree $B$ passing through $p$
with $0<\w(B)\leq \w(A).$ Moreover,
$$
\operatorname{Gwidth}(M, \w)\leq \w(A).
$$
\end{Theorem}
\begin{proof}
Let $c$ be a positive integer such that $[c\w]\in H^2(M; \Z).$ The
first part of the theorem follows from the previous two lemmas when
we apply them to the integral symplectic manifold $(M, c\w).$

Finally, Theorem \ref{nsq} implies the following upper bound for the
Gromov width of the integral symplectic manifold $(M, c\w)$
$$
\operatorname{Gwidth}(M, c\w)\leq c\w(A),
$$
and we are done.
\end{proof}
\begin{Remark}
For a general compact symplectic manifold $(M, \w)$ more involved
constructions are needed to define its Gromov-Witten invariants. In
such constructions, usually one associates to the moduli space of
$J$-holomorphic curves $\overline{\mathcal{M}}_{A, k}(M, J)$ a
\textit{virtual fundamental class} $[\overline{\mathcal{M}}_{A,
k}(M,J)]_{\operatorname{virt}}$ with \textit{rational} coefficients.
The virtual fundamental class
$[\overline{\mathcal{M}}_{A,k}(M,J)]_{\operatorname{virt}}$ is
usually well defined and independent of the almost complex
structure, and Gromov-Witten invariants on $(M, \w)$ are defined by
integrating over this fundamental class (see for example B. Chen,
A.M. Li and B. L. Wang \cite{chenliwang},  Fukaya and Ono
\cite{fukayaono}, Fukaya, Ohta, Oh and Ono \cite{fooo2},
\cite{fooo3}, \cite{fooo}, \cite{fooo4}, Hofer, Wysocki and Zehnder
 \cite{hofer1}, \cite{hoferi}, \cite{hofergw}, \cite{hofer3}, \cite{hofersc}, \cite{hofer2},
 J. Li and G. Tian \cite{litian}, G. Liu and G. Tian
\cite{liutian}, \cite{liutian2}, G. Lu and G. Tian \cite{lutian},
McDuff and Wehrheim \cite{mcduffkuranishiatlases},
\cite{mcduff-wehrheim}, \cite{mcduff-wehreim2}, Pardon
\cite{pardon}, Ruan \cite{ruan}, Siebert \cite{siebert}).

It is expected that many of the applications of Gromov-Witten
invariants in symplectic topology that work for the semipositive
case can be extended to the general case when virtual fundamental
classes are used to define them. For instance, G. Lu has shown with
G. Liu and G. Tian's definition of Gromov-Witten invariants that the
Gromov width of a general compact symplectic manifold is bounded
from above by the symplectic area of a spherical class that has a
non-vanishing Gromov-Witten invariant with one of its constrains
being Poincar\'e dual to the class of a point \cite{glu}[Section
1.5].
\end{Remark}

\section{Coadjoint Orbits of Compact Lie groups}

In this section we recall some general facts about homogeneous
spaces $G_{\mathbb{C}}/P,$ coadjoint orbits and its geometry. Most
of the material shown here can be found in the classical literature
such as Bernstein, Gelfand and Gelfand \cite{GelfandBernstein} and
Kirillov \cite{orbit}. The material presented in Section
\ref{chernline} about Chern classes and stable curves is adapted
from Fulton and Woodward \cite[Chapters 2, 3]{fultonw}.

\subsection{Kostant-Kirillov-Souriau form}

Let $G$ be a connected compact Lie group, $\mathfrak{g}$ be its Lie
algebra, and $\mathfrak{g}^*$ be the dual of the Lie algebra
$\mathfrak{g}$. The compact Lie group $G$ acts on its Lie algebra
$\mathfrak{g}^*$ by the coadjoint action. Let $\lambda \in
\mathfrak{g}^*$ and $\mathcal{O}_\lambda$ be the coadjoint orbit
through $\lambda.$ As a homogeneous space $\mathcal{O}_\lambda \cong
G/G_\lambda,$ where $G_\lambda \subset G$ denotes the stabilizer of
$\lambda\in \mathfrak{g}^*$ under the coadjoint action.

The coadjoint orbit $\mathcal{O}_\lambda$ carries a symplectic form
defined as follows: for $\lambda \in \mathfrak{g}^*$ we define a
skew bilinear form on $\mathfrak{g}$ by
\begin{align*}
\w_{\lambda}: \mathfrak{g}\otimes \mathfrak{g} &\to \R \label{KKS}\\
(X, Y) &\mapsto \langle\lambda, [X, Y] \rangle \nonumber.
\end{align*}
The kernel of $\w_\lambda$ is the Lie algebra $\mathfrak{g}_\lambda$
of the stabilizer $G_\lambda$ of $\lambda \in \mathfrak{g}^*.$ In
particular, $\w_\lambda$ defines a nondegenerate skew-symmetric
bilinear form on $\mathfrak{g}/\mathfrak{g}_\lambda,$ a vector space
that can be identified with $T_\lambda(\mathcal{O}_\lambda)\subset
\mathfrak{g}^*.$ Using the coadjoint action, the bilinear form
$\w_\lambda$ induces a closed, invariant, nondegenerate 2-form on
the orbit $\mathcal{O}_\lambda,$ therefore defining a symplectic
structure on $\mathcal{O}_{\lambda}.$ This symplectic form is known
as the \textbf{Kostant-Kirillov-Souriau form} of the coadjoint orbit
$\mathcal{O}_{\lambda}.$

The compact Lie group $G$ admits a \textbf{complexification}
$G_{\mathbb{C}}.$ There exists a parabolic subgroup $P\subset
G_{\mathbb{C}}$ (see definition below) such that the homogeneous
spaces $G/G_\lambda$ and $G_{\mathbb{C}}/P$ are diffeomorphic. We
can use the complex structure defined on the quotient of complex Lie
groups $G_{\mathbb{C}}/P$ to define a complex structure $J$ on the
coadjoint orbit $G/G_\lambda \cong \mathcal{O}_\lambda$ which is
invariant under the $G_{\mathbb{C}}$-action. Together with the
Kostant-Kirillov-Souriau form, this makes the coadjoint orbit
$\mathcal{O}_\lambda$ a K\"{a}hler manifold.

\subsection{Parabolic subgroups}

Let $T\subset G$ be a maximal torus and $\mathfrak{t}$ denote its
Lie algebra. Let $R \subset \mathfrak{t}^*$ be the root system of
$T$ in $G$ so
$$
\mathfrak{g}_{\mathbb{C}}=\mathfrak{t}_{\mathbb{C}} \oplus
\bigoplus_{\alpha \in R}\mathfrak{g}_\alpha,
$$
where
$$
\mathfrak{g}_\alpha:=\{x\in \mathfrak{g}_{\mathbb{C}}:[\,h\,,
x\,]=\alpha(h)\,x \, \text{ for all }h\in
\mathfrak{t}_{\mathbb{C}}\}$$ is the root space associated with the
root $\alpha \in R.$

Let $R^{+}\subset R$ be a choice of positive roots with simple roots
$S \subset R^+.$ Let $W:=N_{G}(T)/T$ be the \textbf{Weyl group} of
$G.$ For every root $\alpha\in R,$ let $s_\alpha \in W$ be the
reflection associated to it.

Let $B\subset G_{\mathbb{C}}$ be the Borel subgroup with Lie algebra
$$
\mathfrak{b}=\mathfrak{t}_{\mathbb{C}}\oplus \bigoplus_{\alpha \in
R^+}\mathfrak{g}_{\alpha}
$$
We call a subgroup $P\subset G_{\mathbb{C}}$ \textbf{parabolic} if
$B\subset P.$ Let us fix a parabolic subgroup $P\supset B.$ Let
$W_P:=N_P(T)/T$ be the \textbf{Weyl group of $P$} and
$S_P:=\{\alpha\in S:s_\alpha\in W_P\}\subset S$ be the set of simple
roots whose corresponding reflections are in $W_P.$ The group $W_P$
is the subgroup of $W$ generated by the set of simple reflections
$\{s_\alpha:\alpha\in S_P\}.$ The parabolic subgroup $P$ is
generated by the Borel subgroup $B$ and $N_P(T).$ Set $R_P=R\cap \Z
S_P$ and $R_P^+=R^+\cap \Z S_P,$ where $\Z
S_P=\operatorname{span}_{\mathbb{Z}}(S_P)$ is the Abelian group
spanned by $S_P$ in $\mathfrak{t}^*.$ The Lie algebra of $P$ is
$$
\mathfrak{p}=\mathfrak{b}\oplus \bigoplus_{\alpha\in R_P^+}
\mathfrak{g}_{-\alpha}
$$
\begin{Remark}
The map $\tilde{P}\mapsto S_{\tilde{P}}$ establishes a bijection
between the set of all parabolic subgroups of $G_{\mathbb{C}}$
containing $B$ and the set of all subsets of simple roots contained
in $S$ (see for instance Kumar \cite[Chapter 5]{kacmoody}).
\end{Remark}

\subsection{Schubert varieties in $G_{\mathbb{C}}/P$}

For each $w\in W,$ the \textbf{length} $l(w)$ of $w$ is defined as
the minimum number of simple reflections $s_\alpha\in W, \alpha \in
S, $ whose product is $w.$

For $w', w\in W,$ write $w' \to w$ if there exists simple
reflections $s\in S$ such that $w=w'\cdot s$ and $l(w)=l(w')+1.$
Then define $w' \leq_B w$ if there is a sequence
$$
w'\to w_1 \to  \ldots \to w_m=w.
$$
The \textbf{Bruhat order} on $W$ is the partial ordering defined by
the relation $\leq_B.$

Let $W^P\subset W$ be the set of all \textbf{minimum length
representatives} for cosets in $W/W_P.$ Each element $w\in W$ can be
written uniquely as $w=w^Pw_P$ where $w^P\in W^P$ and $w_P\in W_P$.
Their lengths satisfy $l(w)=l(w^P)+l(w_P)$ (see e.g. Humphreys
\cite{humphreys}[Chapter 1]). The \textbf{Bruhat order} $\leq_B$ on
$W^P$ is the restriction to $W^P$ of the Bruhat order on $W.$ The
\textbf{Bruhat order} $\leq_B$ on $W/W_P$ is defined by $w'W_P
\leq_B wW_P$ if and only if $w'^P\leq_B w^P$ on $W^P.$

Let $w_0$ be the longest element in $W$ and let
$B^{\operatorname{op}}:=w_0Bw_0\subset G_{\mathbb{C}}$ be the
\textbf{Borel subgroup opposite} to $B.$ For $w\in W^P$ we define
the \textbf{Schubert cell}
$$
C_P(w):=BwP/P \subset G_{\mathbb{C}}/P
$$
and the \textbf{opposite Schubert cell}
$$
C_P^{\operatorname{op}}(w):=B^{\operatorname{op}}wP/P \subset
G_{\mathbb{C}}/P
$$
The \textbf{Schubert variety} $X_P(w)$ and its opposite
$X_P^{\operatorname{op}}(w)$ are by definition the closures of the
Schubert cells $C_P(w)$ and $C_P^{\operatorname{op}}(w),$
respectively.

The Bruhat order can be written in terms of the inclusion relation
of Schubert varieties, i.e., for $w', w\in W^P,$
$$
X_P(w') \subset X_P(w)
$$
if and only if $w'\leq_B w.$ Indeed,
$$
X_P(w)=\bigsqcup_{w'\leq_B w}C_P(w')
$$
For $w\in W^P,$ the Schubert cell $C_P(w)$ is isomorphic to an
affine space of complex dimension equal to $l(w).$ The set of
Schubert cells $\{X_P(w)\}_{w\in W^P}$ defines a CW-complex for
$G_{\mathbb{C}}/P$ with cells occurring only in even dimension.
Thus, the set of fundamental classes $\{\sigma_P(w):=[X_P(w)]\}_{
w\in W^P}$ is a free basis of $H_*(G_{\mathbb{C}}/P; \Z)$ as a
$\Z$-module. Likewise, the set of cohomology classes
$\{\operatorname{PD}(\sigma_P(w))\}_{w\in W^P}$ is a free basis of
$H^*(G_{\mathbb{C}}/P; \Z)$ as a $\Z$-module. Similar statements
hold for the fundamental classes of the opposite Schubert varieties
$\check{\sigma}_P(w):=[X_P^{\operatorname{op}}(w)]\in
H_*(G_{\mathbb{C}}/P; \Z).$ Note that
$\check{\sigma}_P(w)=\sigma_P(\check{w})$ where
$\check{w}:=w_0ww_p\in W^P$ and $w_p$ denotes the longest element in
$W_P.$

The Poincar\'e intersection pairing is the map
\begin{align*}\langle \,,\, \rangle: H_*(G_{\mathbb{C}}/P; \Z)\otimes
H_*(G_{\mathbb{C}}/P; \Z)\to \Z
\end{align*}
that associates to a pair of homology classes $a, b\in
H_*(G_{\mathbb{C}}/P; \Z)$ the coefficient at the class of a point
$[\operatorname{pt}]$ of the homological intersection product
$a\cdot b\in H_*(G_{\mathbb{C}}/P; \Z).$ The \textbf{duality
Theorem}  states that for any $w', w\in W^P,$
$$
\langle \check{\sigma}_P(w), \sigma_P(w') \rangle=\delta_{ww'},
$$
and in particular
$$
\dim_{\mathbb{C}}(G_{\mathbb{C}}/P)=\dim_{\mathbb{C}}(X_P(w))+\dim_{\mathbb{C}}(X_P^{\operatorname{op}}(w))
$$

\subsection{$T$-invariant curves}

The collection of cosets $\{w\cdot P\}_{w\in W^P}$ is the set of all
$T$-fixed points in $G_{\mathbb{C}}/P.$ For each positive root
$\alpha \in R^+-R^+_P$ there is a unique irreducible $T$-invariant
curve $C_\alpha$ that contains $1\cdot P$ and $s_\alpha\cdot P.$
Indeed, $C_\alpha:=Sl(2, \C)_\alpha\cdot P/P$ where $Sl(2,
\C)_\alpha \subset G_{\mathbb{C}}$ is the subgroup of
$G_{\mathbb{C}}$ with Lie algebra $\mathfrak{g}_{\alpha}\oplus
\mathfrak{g}_{-\alpha}\oplus [\mathfrak{g}_{\alpha},
\mathfrak{g}_{-\alpha}].$ The curve $C_\alpha$ is unique because
there is a neighborhood of $1\cdot P/P$ that is $T$-equivariantly
isomorphic to $\mathfrak{g}_{\mathbb{C}}/\mathfrak{p}.$ The
$T$-invariant curves in $\mathfrak{g}_{\mathbb{C}}/\mathfrak{p}$
correspond to weight spaces $\mathfrak{g}_{-\alpha}$ for $\alpha \in
R^+-R^+_P.$

\subsection{Chern classes}\label{chernline}

Let $(\cdot\, , \cdot)$ denote an ad-invariant inner product defined
on $\operatorname{Lie}(G)=\mathfrak{g}.$  We identify the Lie
algebra $\mathfrak{g}$ and its dual $\mathfrak{g}^*$ via this inner
product. The inner product $(\cdot\, , \cdot)$ defines an inner
product in $\mathfrak{t}^*= \R R.$ Each root $\alpha \in R$ has a
\textbf{coroot} $\check{\alpha} \in \mathfrak{t}$ that is identified
with $\dfrac{2\alpha}{(\alpha, \alpha)}$ via the inner product
$(\cdot\, , \cdot).$ The coroots form the dual root system
$\check{R}=\{\check{\alpha}:\alpha \in R\},$ with basis of simple
coroots $\check{S}=\{\check{\alpha}:\alpha \in S\}.$ For the
parabolic subgroup $P\subset G_{\mathbb{C}},$ let
$\check{S}_P:=\{\check{\alpha}:\alpha\in S_P\}\subset \check{S}.$

The \textbf{fundamental weight} $\w_{\alpha}\in \mathfrak{t}^*$
associated with $\alpha\in S$ is defined by $(\w_\alpha,
\check{\beta})=\delta_{\alpha, \beta}$ for $\beta\in S.$ A
\textbf{weight} is an element in the Abelian group spanned by the
set of fundamental weights.

The cohomology group $H^2(G_{\mathbb{C}}/P; \Z)$ can be identified
with the span
$$
\Z\{\w_\alpha:\alpha\in S-S_P\}
$$
and the homology group $H_2(G_{\mathbb{C}}/P; \Z)$ with the quotient
$$
\Z\check{S}/\Z\check{S}_P
$$
For each $\alpha \in S-S_P$ we identify the Schubert class
$\sigma_P(s_\alpha)\in H_2(G_{\mathbb{C}}/P; \Z)$ with
$\check{\alpha}+\Z\check{S}_P\in \Z\check{S}/\Z\check{S}_P$ and the
Poincar\'e dual class
$\operatorname{PD}(\check{\sigma}_P(s_\beta))\in
H^2(G_{\mathbb{C}}/P; \Z)$ with $\w_\beta.$ The pairing
\begin{align*}
H^2(G_{\mathbb{C}}/P; \Z)&\otimes H_2(G_{\mathbb{C}}/P; \Z) \to \Z \\
(\sigma, \alpha) &\mapsto \int_\sigma \alpha
\end{align*}
is then given by the ad-invariant inner product $(\cdot \,, \cdot)$
on $\mathfrak{t}.$

The following localization lemma, due to Bott \cite{bott}, allows us
to compute the first Chern classes of line bundles on $T$-invariant
curves of $G_{\mathbb{C}}/P:$

\begin{lemma}
Suppose that a torus $T$ acts on a curve $C\cong \cpi,$ with fixed
points $p\ne q,$ and suppose $L$ is a $T$-equivariant line bundle on
$C.$ Let $\eta_p$ and $\eta_q$ be the weights of $T$ acting on the
fibers $L_p$ and $L_q,$ and let $\psi_p$ be the weight of $T$ acting
on the tangent space to $C$ at $p.$ Then
$$
\eta_p-\eta_q=n\psi_p
$$
where $n=\int_C c_1(L)$ is the degree of $L.$
\end{lemma}

If $\eta$ is a weight that vanishes on all $\beta$ in $S_P,$ it
determines a character on $P,$ and so a line bundle
$L(\eta):=G_{\mathbb{C}}\times_P \C(\eta)$ over $G_{\mathbb{C}}/P.$
The Chern class $c_1(L(\eta))\in H^2(G_{\mathbb{C}}/P; \Z)$  is
identified with the weight $\eta\in \Z\{w_\alpha:\alpha\in S-S_P\}.$
Indeed, if $L$ is any holomorphic line bundle over
$G_{\mathbb{C}}/P,$ there exists a unique weight $\eta\in
\Z\{w_\alpha:\alpha\in  S-S_P\}$ such that $L\cong L(\eta),$ and
$\operatorname{Pic}(G_{\mathbb{C}}/P)\cong \Z\{w_\alpha:\alpha\in
S-S_P\}\cong H^2(G_{\mathbb{C}}/P; \Z)$ (see e.g. Borel-Weil
\cite{borelweil}). 

\begin{Proposition}
For any root $\alpha\in R^+-R^+_P,$
$$
[C_\alpha]=\check{\alpha}+\Z\check{S}_P \in H_2(G_{\mathbb{C}}/P;
\Z) \cong \Z\check{S}/\Z\check{S}_P
$$
\end{Proposition}
\begin{proof}

The localization lemma implies that for any weight $\eta\in
\Z\{w_\alpha:\alpha\in S-S_P\}$
$$
\int_{C_\alpha}c_1(L(\eta))\cdot
(-\eta)=-\eta-(-s_\alpha(\eta))=s_{\alpha}(\eta)-\eta
$$
and thus
$$
(\eta\,, \check{\alpha})=\int_{C_\alpha}c_1(L(\eta))
$$
The nondegeneracy of the pair $(\,, \,)$ implies the proposition.

\end{proof}

\begin{Proposition}\label{chern}
The Chern class $c_1(T(G_{\mathbb{C}}/P))$ is identified with
$\sum_{\gamma \in R^{+}-R^+_P}\gamma$ via the isomorphism
\begin{align*}
\Z\{w_\alpha:\alpha\in  S-S_P\} &\to H^2(G_{\mathbb{C}}/P;
\Z)\\
\eta &\mapsto c_1(L(\eta))
\end{align*}
\end{Proposition}
\begin{proof}
The tangent space of $G_{\mathbb{C}}/P$ at the point $1\cdot P\in
G_{\mathbb{C}}/P$ is
$$
\mathfrak{g}_{\mathbb{C}}/\mathfrak{p}=\bigoplus_{\alpha \in
R^+-R^+_P}\mathfrak{g}_{-\alpha}
$$
The weight of the line bundle $\bigwedge^n T(G_{\mathbb{C}}/P), \,
n=\dim_{\mathbb{C}}(G_{\mathbb{C}}/P), \,$ at the point $1\cdot P\in
G_{\mathbb{C}}/P$ is $-\sum_{\gamma \in R^{+}-R^+_P}\gamma.$ The
proposition follows from the fact that
$c_1(T(G_{\mathbb{C}}/P))=c_1\Bigl(\bigwedge^n
T(G_{\mathbb{C}}/P)\Bigr).$
\end{proof}

\subsection{Pl\"{u}cker embedding}
Let
$$
\eta=\sum_{\beta\in S-S_P}l_\beta w_\beta\in \Z_{\geq
0}\{w_\alpha:\beta\in S-S_P\},
$$
be an integral weight and $V_\eta$ be the irreducible representation
of $G_{\mathbb{C}}$ with highest weight $\eta.$ The Borel-Weil-Bott
Theorem states that the set of holomorphic sections
$H^0(G_{\mathbb{C}}/P, L(\eta))$ of the line bundle $L(\eta)$ is
isomorphic as a $G_{\mathbb{C}}$-representation to the irreducible
representation $V_\eta.$

Let $v_\eta$ be a highest weight vector of $V_\eta$. We can embed
$G_{\mathbb{C}}/P$ in the projective space $\mathbb{P}V_\eta$ by the
transformation
\begin{align*}
G_{\mathbb{C}}/P &\hookrightarrow \mathbb{P}V_\eta \nonumber \\
[g] &\mapsto [g\cdot v_\eta]
\end{align*}

\begin{Remark}
For $\lambda \in \R_{\geq 0}\{w_\alpha:\beta\in S-S_P\}\subset
\mathfrak{t}^*,$ let $\mathcal{O}_\lambda$ be the coadjoint orbit
passing through $\lambda$ and $\w_\lambda$ its
Kostant-Kirillov-Souriau form. The cohomology class of $\w_\lambda$
is identified with $\lambda \in H^2(G_{\mathbb{C}}/P; \, \R)\cong \R
\{w_\alpha:\beta\in S-S_P\}\subset \mathfrak{t}^*,$ and for any
positive root $\alpha \in R^+-R^+_P$
$$
\w_\lambda(C_\alpha)=\int_{C_\alpha} \w_\lambda=\langle \lambda\, ,
\check{\alpha}\rangle
$$
If $\lambda$ is integral, the projective embedding $G_{\mathbb{C}}/P
\hookrightarrow \mathbb{P}V_\lambda$ is symplectic, i.e., the
pullback of the Fubini-Study form defined on $\mathbb{P}V_\lambda$
is the Kostant-Kirillov-Souriau form $\w_\lambda$ defined on
$\mathcal{O}_\lambda\cong G_{\mathbb{C}}/P.$

\end{Remark}

\subsection{Algebraic Gromov-Witten invariants}

Let $J$ be the invariant complex structure defined on the quotient
of complex Lie groups $G_{\mathbb{C}}/P.$ Let $A\in
H_2(G_{\mathbb{C}}/P; \Z)$ and $k\in \Z_{\geq 0}.$  The moduli space
$\overline{\mathcal{M}}_{A, k}(G_{\mathbb{C}}/P, J)$ of stable
$J$-holomorphic curves of degree $A$ with $k$-marked points is a
normal projective variety of dimension equal to
$$
\dim \overline{\mathcal{M}}_{A, k}(G_{\mathbb{C}}/P,
J)=\dim(G_{\mathbb{C}}/P)+ 2c_1(A)+2k-6
$$
(see e.g. Fulton and Pandharipande \cite{fultonp}).

Let
$$
\operatorname{ev}_J^k=(\operatorname{ev}_1, \cdots,
\operatorname{ev}_k):\overline{\mathcal{M}}_{A, k}(G_{\mathbb{C}}/P,
J) \to (G_{\mathbb{C}}/P)^k
$$
be the evaluation map. The \textbf{algebraic Gromov-Witten
invariant} $\operatorname{GW}_{A, k}^{\,\operatorname{alg}}$ on
$G_{\mathbb{C}}/P$ is defined by
$$
\operatorname{GW}_{A,k}^{\,\operatorname{alg}}(a_1, \cdots,
a_k):=\int_{\overline{\mathcal{M}}_{A, k}(G_{\mathbb{C}}/P,
J)}\operatorname{ev}_1^*a_1\cup \cdots \cup \operatorname{ev}_k^*a_k
$$
for cohomology classes $a_1, \cdots, a_k\in H^*(G_{\mathbb{C}}/P;
\Z).$ If $X_1,\ldots, X_k \subset G_{\mathbb{C}}/P$ are irreducible
varieties whose fundamental classes are Poincar\'e dual to
$a_1,\ldots, a_k \in H^*(G_{\mathbb{C}}/P; \Z)$ and
$$
\dim{\overline{\mathcal{M}}_{A, k}(G_{\mathbb{C}}/P,
J)}=\sum_{i=1}^k \deg{a_i},$$ the number of $J$-holomorphic spheres
of degree $A$ passing through $g_1X_1, \cdots, g_kX_k$ coincides
with the Gromov--Witten invariant $\operatorname{GW}_{A,
k}^{\,\operatorname{alg}}(a_1,\ldots, a_k)$ for generic $g_1,
\ldots,\\ g_k \in G_{\mathbb{C}}$ (see e.g. Fulton and Pandharipande
\cite[Lemma 14]{fultonp}).

\begin{Remark}
Let $\mathcal{O}_\lambda$ be a coadjoint orbit passing through
$\lambda\in \mathfrak{g}^*$ with Kostant-Kirillov-Souriau form
$\w_\lambda.$ Assume that $P\subset G_{\mathbb{C}}$ is a parabolic
subgroup such that $\mathcal{O}_\lambda \cong G_{\mathbb{C}}/P.$ Let
$A\in H_2(G_{\mathbb{C}}/P; \Z)$ and $k\in \Z_{\geq 0}.$

If $J\in \mathcal{J}(\mathcal{O}_\lambda, \w_\lambda)$ is a regular
almost complex structure and the class $A$ is $J$-indecomposable
class, we denote by $\operatorname{GW}^{J}_{A, k}$ the Gromov-Witten
invariant defined exclusively for $J$ and $A$ (see Remark
\ref{gwj}). If $(\mathcal{O}_\lambda, \w_\lambda)$ is an integral
coadjoint orbit, we denote by
$\operatorname{GW}^{\operatorname{cm}}_{A, k}$ the symplectic
Gromov-Witten invariant defined by Cielibak and Mohnke (see Section
\ref{gwcm}).

When the coadjoint orbit  $(\mathcal{O}_\lambda, \w_\lambda)$ is
integral, $J$ is the almost complex structure coming from the
quotient of complex Lie groups $G_{\mathbb{C}}/P$ and $A$ is a
$J$-indecomposable class, all the three Gromov-Witten invariants
$$\operatorname{GW}^{\operatorname{alg}}_{A, k},\,\,\,
\operatorname{GW}^{J}_{A, k},\,\,\,
\operatorname{GW}^{\operatorname{cm}}_{A, k}$$ coincide.
\end{Remark}

\section{Curve Neighborhoods and Gromov Witten invariants}

In this section we define the concept of curve neighborhood and
explain its relation with Gromov-Witten invariants. The material
presented here is mostly adapted from Buch and Mihalcea
\cite{curvenghb}.

Let $G$ be a compact connected Lie group. Let $B$ be a Borel
subgroup and $P$ be a parabolic subgroup of $G_{\mathbb{C}}$ with
$B\subset P.$ Let $J$ be the complex structure defined on
$G_{\mathbb{C}}/P$ coming from its presentation as a quotient of
complex Lie groups. Let $A\in H_2(G_{\mathbb{C}}/P; \Z)$ be a
spherical class. Let $\overline{\mathcal{M}}_{A,
2}(G_{\mathbb{C}}/P, J)$ be the moduli space of equivalence classes
of stable $J$-holomorphic curves of degree $A$ with two marked
points and
$$
\operatorname{ev}_J^2=(\operatorname{ev}_1, \operatorname{ev}_2):
\overline{\mathcal{M}}_{A, 2}(G_{\mathbb{C}}/P, J) \to
(G_{\mathbb{C}}/P)^2
$$
be its corresponding evaluation map. Given any subvariety $Z\subset
G_{\mathbb{C}}/P,$ define the \textbf{degree $A$ neighborhood} of
$Z$ to be
$$
\Gamma_A(Z):=\operatorname{ev}_2(\operatorname{ev}_1^{-1}(Z))\subset
G_{\mathbb{C}}/P,
$$
i.e., the union of all stable $J$-holomorphic curves of degree $A$
that meet $Z.$ The \textbf{Gromov-Witten variety} of $Z$ is
$$
\operatorname{GW}_A(Z):=\operatorname{ev}_1^{-1}(Z)\subset
\overline{\mathcal{M}}_{A, 2}(G_{\mathbb{C}}/P, J)
$$
By definition,
$$
\Gamma_A(Z)=\operatorname{ev_2}(\operatorname{GW}_A(Z))
$$
When $Z\subset G_{\mathbb{C}}/P$ is a irreducible variety, then
$\Gamma_A(Z)$ is also a irreducible variety. In particular, if $Z$
is a $B$-stable irreducible variety, i.e. a $B$-stable Schubert
variety, then so is $\Gamma_A(Z)$ (Buch, Chaput, Mihalcea and Perrin
\cite{BuchChaputMihalceaPerrin}).

For our purposes, the following lemma would be enough to compute
Gromov-Witten invariants:
\begin{lemma}\label{alexcastro}
Let $A\in H_2(G_{\mathbb{C}}/P;\,\Z)$ be a $J$-indecomposable class
and $\Gamma_A(\operatorname{pt})$ be the degree $A$ neighborhood of
a point in $G_{\mathbb{C}}/P.$ If
$$
c_1(A)=1+\dim_{\mathbb{C}}{\Gamma_A(\operatorname{pt})},
$$
then
$$
\operatorname{GW}_{A,
2}^{\,\operatorname{alg}}(\operatorname{PD}[\operatorname{pt}],\operatorname{PD}[\Gamma_A(\operatorname{pt})]^{\operatorname{op}})>
0
$$
\end{lemma}
\begin{proof}
The Bertini-Kleiman transversality theorem implies that for generic
$g, h\in G_{\mathbb{C}}$ the evaluation map
$$
\operatorname{ev}_J:\overline{\mathcal{M}}_{A, 2}(G_{\mathbb{C}}/P,
J)\to (G_{\mathbb{C}}/P)^2
$$
is transverse to $\{g\cdot P\} \times h\cdot \Gamma_A(1\cdot
P)^{\operatorname{op}}\subset (G_{\mathbb{C}}/P)^2.$

On the other hand, the duality theorem implies that for generic $g,
h\in G_{\mathbb{C}},$ the Schubert varieties $g\cdot \Gamma_A(1\cdot
P)$ and $h\cdot \Gamma_A(1\cdot P)^{\operatorname{op}}$ intersect
transversally at one point say $q\cdot P.$ In particular, there
exists a stable curve $(u; \Sigma; z_1, z_2)$ with two marked points
such that $u(z_1)=g\cdot P$ and $u(z_2)=q\cdot P\in h\cdot
\Gamma_A(1\cdot P)^{\operatorname{op}}.$ The indecomposability
condition of $A$ implies that $\Sigma\cong \cpi.$

If the opposite Schubert variety
$\Gamma_A(\operatorname{pt})^{\operatorname{op}}$ satisfies the
dimensional constraint
$$
\dim_{\mathbb{C}}{\Gamma_A(\operatorname{pt})^{\operatorname{op}}}+\dim_{\mathbb{C}}{\overline{\mathcal{M}}_{A,2}(G_{\mathbb{C}}/P)}=2\dim_{\mathbb{C}}{G_{\mathbb{C}}/P},
$$
that is the same as having
$c_1(A)=1+\dim_{\mathbb{C}}{\Gamma_A(\operatorname{pt})},$ the
Gromov-Witten invariant $\operatorname{GW}_{A,
2}^{\,\operatorname{alg}}(\operatorname{PD}[\operatorname{pt}],\operatorname{PD}[\Gamma_A(\operatorname{pt})]^{\operatorname{op}})$
is finite and positive (see e.g. McDuff and Salamon
\cite{mcduff}[Proposition 7.4.5]).
\end{proof}
\begin{Remark}
For any two $B$-stable Schubert varieties $Z_1, Z_2$ and any degree
$A\in H_2(G_{\mathbb{C}}/P; \Z)$ the following formula due to Buch
and Mihalcea \cite{curvenghb}
$$
\operatorname{GW}_{A,
2}^{\,\operatorname{alg}}(\operatorname{PD}[Z_1],\operatorname{PD}[Z_2])=
\begin{cases}
1 & \text{ if }c_1(A)-1+\dim_{\mathbb{C}}{Z_1}=\dim_{\mathbb{C}}{\Gamma_A(Z_1)} \\ & \text{ and $[\Gamma_A(Z_1)]$ is the Schubert class} \\
& \text{ opposite to $[Z_2]$} \\
0 & \text{otherwise}
\end{cases}
$$
extends the previous lemma. The above formula is a consequence of
the \textit{projection formula} and the fact that the pushforward of
$\operatorname{PD}[\operatorname{GW}_A(Z_1)]$ under the evaluation
map $\operatorname{ev}_{2}$ is equal to
$\operatorname{PD}[\Gamma_A(Z_1)]$ if
$\dim_{\mathbb{C}}{\operatorname{GW}_A(Z_1)}=\dim_{\mathbb{C}}{\Gamma_A(Z_1)}$
and zero otherwise (Buch and Mihalcea \cite{curvenghb}). We briefly
explain how the above formula is deduced from these two facts:
\begin{align*}
&\operatorname{GW}_{A,
2}^{\,\operatorname{alg}}(\operatorname{PD}[Z_1],\operatorname{PD}[Z_2]):=\int_{\overline{\mathcal{M}}_{A,2}(G_{\mathbb{C}}/P,
J)}\operatorname{ev}_1^*\operatorname{PD}[Z_1]\cup
\operatorname{ev}_2^*\operatorname{PD}[Z_2] \\
&=\int_{G_{\mathbb{C}}/P}\operatorname{ev}_{2*}(\operatorname{ev}_1^*\operatorname{PD}[Z_1])\cup
\operatorname{PD}[Z_2]=\int_{G_{\mathbb{C}}/P}\operatorname{ev}_{2*}(\operatorname{PD}[\operatorname{GW}_A(Z_1)])\cup
\operatorname{PD}[Z_2]\\
&=\begin{cases}
\int_{G_{\mathbb{C}}/P}\operatorname{PD}[\Gamma_A(Z_1)]\cup
\operatorname{PD}[Z_2] & \text{if
}\dim_{\mathbb{C}}{\operatorname{GW}_A(Z_1)}=\dim_{\mathbb{C}}{\Gamma_A(Z_1)} \\
0 & \text{otherwise}
\end{cases}
\end{align*}
\end{Remark}
\begin{example}
Let $G(k, n)$ denote the Grassmannian manifold of $k$-dimensional
vector spaces in $\C^n.$ Let $A$ be the Schubert class that cyclic
generates the homology group $H_2(G(k, n); \Z).$ Let $J$ be the
invariant complex structure defined on $G(k, n).$

For a $J$-holomorphic curve $u:\cpi \to G(k, n)$ of degree $A,$
there exist linearly independent vectors $v_0, v_1, \ldots, v_k\in
\C^n$ such that
\begin{align*}
u:\cpi &\to G(k, n)\\
[z_0, z_1] &\mapsto \operatorname{span}\{z_0v_0+z_1v_1, v_2, \ldots,
v_k\}
\end{align*}
If we let $W^{k-1}=\operatorname{span}\{v_2, \ldots, v_k\}$ and
$W^{k+1}=\operatorname{span}\{v_0, v_1, v_2, \ldots, v_k\},$ then
$$
u(\cpi)=\{V^k\in G(k, n): W^{k-1}\subset V^k \subset W^{k+1}\}
$$
Note that two vector subspaces $V^k, V'^{k}\in G(k, n)$ can be
joined by a $J$-holomorphic curve of degree $A$ if and only if
$$
\dim(V^k\cap V'^{k})\geq k-1\,\, \text{ and }\,\,
\dim(V^k+V'^{k})\leq k+1.
$$
Thus, for a fixed vector subspace $V^k\in G(k, n),$ the degree $A$
neighborhood of $V^k$ is the set
$$
\Gamma_A(V^k)=\{V'^{k}\in G(k, n):k-1\leq \dim(V^k\cap V'^{k})\leq
\dim(V^k + V'^{k})\leq k+1\}
$$
The complex dimension of $\Gamma_A(V^k)$ equals to $n-1=c_1(A)-1.$
By the previous Lemma
$$
\operatorname{GW}_{A,
2}^{\,\operatorname{alg}}(\operatorname{PD}[\operatorname{pt}],\operatorname{PD}[\Gamma_A(\operatorname{pt})]^{\operatorname{op}})=1
$$
\end{example}

The following Lemma allows us to give an explicit description of the
degree $A$ neighborhood $\Gamma_A(1\cdot P)$ of the point $1\cdot
P\in G_{\mathbb{C}}/P$ when $P$ is a maximal parabolic subgroup of
$G_{\mathbb{C}}$ and $A$ is the Schubert class that cyclic generates
$H_2(G_{\mathbb{C}}/P; \Z).$
\begin{lemma}\label{fixedpoints}
Let $\alpha \in S$ and $P\subset G_{\mathbb{C}}$ be the maximal
parabolic subgroup such that $S_P=S-\{\alpha\}.$ Let
$A=\sigma_P(s_\alpha)$ be the Schubert class that cyclic generates
$H_2(G_{\mathbb{C}}/P; \Z).$ Let
\begin{align*}
Z_A^{P}:=\{s_\beta \cdot P\in W/W_P: \beta \in R^+-R_P^+, \,\,
\check{\beta}=\check{\alpha}+ \Z \check{S}_P\}
\end{align*}
The set $Z_A^P$ has a unique maximal element $z_A^P$ with respect to
the Bruhat order defined on $W/W_P$ and
$$
\Gamma_A(1\cdot P)=X_P(z_A^P)
$$
\end{lemma}
\begin{proof}
The curve neighborhood $\Gamma_{A}(1\cdot P)$ is a $B$-stable
Schubert variety and thus determined by the $T$-fixed points that it
contains. Note that if there exists a $J$-holomorphic curve of
degree $A$ passing through two $T$-fixed points, the curve is
$T$-invariant. Hence, the set of $T$-fixed points in
$\Gamma_{A}(1\cdot P)$ is the set of all $T$-fixed points that can
be joined with $1\cdot P$ by a $T$-invariant curve of degree $A.$
This corresponds to the set
\begin{align*}
Z_A^{P}=\{s_\beta \cdot P\in W/W_P: \beta \in R^+-R_P^+, \,\,
\check{\beta}=\check{\alpha}+ \Z \check{S}_P\}
\end{align*}
(see Section \ref{chernline}). The set $Z_A^P$ has a unique maximal
element $z_A^P$ with respect to the Bruhat order because
$\Gamma_A(1\cdot P)$ is a Schubert variety and thus
$$
\Gamma_A(1\cdot P)=X_P(z_A^P) \qedhere
$$
\end{proof}

\section{Upper Bound for the Gromov width of Grassmannian
Manifolds}\label{generalizedgrassmannian}

Let $G$ be a compact connected simple Lie group with Lie algebra
$\mathfrak{g}.$ Let $T\subset G$ be a maximal torus and let
$B\subset G_{\mathbb{C}}$ be a Borel subgroup with
$T_{\mathbb{C}}\subset B\subset P.$ Let $W=N(T)/T$ be the associated
Weyl group. Let $R$ be the set of roots and $S$ be the choice of
simple roots compatible with $B.$ For a parabolic subgroup $P\subset
G_{\mathbb{C}}$ that contains $B,$ let $W_P$ be the Weyl group of
$P$ and $S_P$ be the subset of simple roots whose corresponding
reflections are in $W_P.$

The \textbf{maximal parabolic subgroup of $G_{\mathbb{C}}$
associated with a simple root $\alpha\in S$} is the parabolic
subgroup $P_\alpha$ such that $S_{P_\alpha}=S-\{\alpha\}.$  We call
the corresponding homogeneous space $G_{\mathbb{C}}/P_\alpha$ a
\textbf{Grassmannian manifold}. The second homology group
$H_2(G_{\mathbb{C}}/P_\alpha;\,\Z)$ is generated as a $\Z$-module by
the fundamental class $[X_{P_\alpha}(s_\alpha)].$ We denote the
class $[X_{P_\alpha}(s_\alpha)]$ by $A.$

Let $\lambda\in \mathfrak{t}^*\subset \mathfrak{g}^*.$ Let us assume
that the coadjoint orbit $\mathcal{O}_\lambda$ passing through
$\lambda$ is isomorphic with the Grassmannian manifold
$G_{\mathbb{C}}/P_\alpha$ for some $\alpha\in S.$ In this section we
will show that
$$
\operatorname{Gwidth}(\mathcal{O}_\lambda, \w_\lambda)\leq |\langle
\lambda\,, \check{\alpha} \rangle|,
$$
where $\w_\lambda$ denotes the Kostant-Kirillov-Souriau form defined
on $\mathcal{O}_\lambda.$ We obtain this upper bound by computing a
non-vanishing Gromov-Witten invariant with one of its constraints
being Poincar\'e dual to a point. More precisely, we show that if
$\Gamma_A(\operatorname{pt})$ is the degree $A$ neighborhood of a
point in $G_{\mathbb{C}}/P_\alpha,$ then
$$
\operatorname{GW}_{A,
2}^{\,\operatorname{alg}}(\operatorname{PD}[\operatorname{pt}],\operatorname{PD}[\Gamma_A(\operatorname{pt})]^{\operatorname{op}})>
0
$$
First, we give an explicit description of the degree $A$
neighborhood $\Gamma_A(1\cdot P_\alpha)$ of the point $1\cdot
P_\alpha$ when $P_\alpha\subset G_{\mathbb{C}}$ is a maximal
parabolic subgroup associated with a \textit{long} simple root
$\alpha\in S.$
\begin{Theorem}
Let $\alpha \in S$ be a simple root, $P \subset G_{\mathbb{C}}$ be
the maximal parabolic subgroup associated with $\alpha$ and $A$ be
the Schubert class that cyclic generates $H_2(G_{\mathbb{C}}/P;
\Z).$ Let $N(\alpha)\subset S$ denotes the neighbors of $\alpha$ in
the Dynkin diagram of $G.$ Let $R\subset G_{\mathbb{C}}$ be the
parabolic subgroup with $S_{R}=S-(N(\alpha)\cup \{\alpha\}).$ If
$\alpha$ is a long simple root, then the degree $A$ neighborhood
$\Gamma_A(1\cdot P)$ of the point $1\cdot P \in G_{\mathbb{C}}/P$ is
a $B$-stable Schubert variety and
$$
\Gamma_A(1\cdot P)=X_{P}(w_p^{r}s_\alpha),
$$
where $w_p^{r}$ is the longest element in the set of minimum length
representatives of cosets in $W_P/W_R.$
\end{Theorem}
\begin{proof}
Let $J$ be the invariant complex structure defined on
$G_{\mathbb{C}}/P.$ For any nonnegative integer $k,$ we denote by
$\overline{\mathcal{M}}_{A,k}(G_{\mathbb{C}}/P, J)$ the moduli space
of equivalent classes of stable $J$-holomorphic curves of degree $A$
with $k$-marked points.

Let $f:\overline{\mathcal{M}}_{A,1}(G_{\mathbb{C}}/P, J)\to
\overline{\mathcal{M}}_{A,0}(G_{\mathbb{C}}/P, J)$ be the forgetful
map that maps a class $[u; (\cpi; z)]$  to $[u]$ and
$\operatorname{ev}_J^1:\overline{\mathcal{M}}_{A,1}(G_{\mathbb{C}}/P,
J)\to G_{\mathbb{C}}/P$ be the evaluation map that maps a class $[u;
(\cpi; z)]$ to $u(z).$ We have a diagram of arrows
$$
\centering
\begin{diagram}
\node{\overline{\mathcal{M}}_{A,1}(G_{\mathbb{C}}/P, J)} \arrow{e,t}{f} \arrow{s,l}{\operatorname{ev}_J^1} \node{\overline{\mathcal{M}}_{A,0}(G_{\mathbb{C}}/P, J)}\\
\node{G_{\mathbb{C}}/P} \\
\end{diagram}
$$
Note that the curve neighborhood $\Gamma_A(1\cdot P)$ of the point
$1\cdot P\in G_{\mathbb{C}}/P$ is the same as
\begin{align*}
\Gamma_A(1\cdot
P)=\operatorname{ev}_{J*}^1(f^{*}(f_*(\operatorname{ev}^{1*}_{J}(1\cdot
P))))
\end{align*}
Let $Q\subset G_{\mathbb{C}}$ be the parabolic subgroup with
$S_{Q}=S-N(\alpha).$ The complex group $G_{\mathbb{C}}$ acts
holomorphically on $G_{\mathbb{C}}/P$ and trivially on
$H_*(G_{\mathbb{C}}/P; \Z),$ as a consequence there is a group
action of $G_{\mathbb{C}}$ on
$\overline{\mathcal{M}}_{A,k}(G_{\mathbb{C}}/P, J).$ This action is
transitive when $k=0,1$ and $\alpha$ is a long simple root. Under
this action, the moduli spaces
$\overline{\mathcal{M}}_{A,0}(G_{\mathbb{C}}/P, J),\,
\overline{\mathcal{M}}_{A,1}(G_{\mathbb{C}}/P, J)$ are isomorphic to
the homogeneous spaces $G_{\mathbb{C}}/Q,\, G_{\mathbb{C}}/R,$
respectively. Via these isomorphisms, the diagram of arrows shown
above is compatible with the following diagram
\begin{center}
$$
\begin{diagram}
\node{G_{\mathbb{C}}/R} \arrow{e,t}{\pi_q} \arrow{s,l}{\pi_p} \node{G_{\mathbb{C}}/Q}\\
\node{G_{\mathbb{C}}/P} \\
\end{diagram}
$$
\end{center}
where $\pi_p$ and $\pi_q$  denote the projection quotient maps (see
e.g. Manivel and Landsberg \cite{landsbergmanivel}, Strickland
\cite{strickland}). Thus,
\begin{align*}
\Gamma_A(1\cdot P)=\pi_{p*}(\pi_q^{*}(\pi_{q*}(\pi_p^{*}(1\cdot
P))))
\end{align*}
From Lemma \ref{echeck} in the appendix, we have that
$$
\Gamma_A(1\cdot
P)=\pi_{p*}(\pi_q^{*}(X_{Q}(w_p^{r})))=X_P(w_p^{r}s_\alpha),
$$
and we are done.
\end{proof}
Now we show that when $P\subset G_{\mathbb{C}}$ is a maximal
parabolic subgroup and $A$ is the class that cyclic generates
$H_2(G_{\mathbb{C}}/P; \Z),$ the Gromov-Witten invariant
$\operatorname{GW}_{A,
2}^{\,\operatorname{alg}}\bigl(\operatorname{PD}[1\cdot
P],\operatorname{PD}[\Gamma_A(1\cdot P)^{\operatorname{op}}]\bigr) $
is different from zero.

\begin{Theorem}\label{alexcastro2}
Let $P\subset G_{\mathbb{C}}$ be a maximal parabolic subgroup, $A$
be the Schubert class that cyclic generates $H_2(
G_{\mathbb{C}}/P;\, \Z)$ and $\Gamma_A(1\cdot P)$ be the degree $A$
neighborhood of $1\cdot P.$ Then
$$
\operatorname{GW}_{A,2}^{\,\operatorname{alg}}\bigl(\operatorname{PD}[1\cdot
P],\operatorname{PD}[\Gamma_A(1\cdot P)^{\operatorname{op}}]\bigr)=1
$$
\end{Theorem}
\begin{proof}
According to Lemma \ref{alexcastro}, it is enough to check that the
curve neighborhood $\Gamma_A(1\cdot P)$ satisfies the dimensional
constrain
$$
c_1(A)=1+\dim_{\mathbb{C}}(\Gamma_A(1\cdot P))
$$
We split the proof in several cases:

\setlength{\leftmargini}{1em}
\begin{itemize}

\item Long root case: Assume that $P\subset G_{\mathbb{C}}$ is a maximal parabolic subgroup associated with a
long simple root $\alpha\in S.$ Let $R\subset G_{\mathbb{C}}$ be the
parabolic subgroup with $S_{R}=S-(N(\alpha)\cup \{\alpha\}),$ where
$N(\alpha)\subset S$ denotes the neighbors of $\alpha$ in the Dynkin
diagram of $G.$ By the previous Theorem and Lemma \ref{echeck} in
the appendix,
$$
\dim_{\mathbb{C}}{\Gamma_A(1\cdot P)}=l(w_{p}^{r}s_\alpha),
$$
where $w_p^{r}$ is the longest element in the set of minimum length
representatives of cosets in $W_P/W_R.$ We have that
\begin{align*}
l(w_{p}^{r}s_\alpha)&=l(w_{p}^{r})+1=\dim_{\mathbb{C}}(G_{\mathbb{C}}/R)-\dim_{\mathbb{C}}(G_{\mathbb{C}}/P)+1
\\&=\dim_{\mathbb{C}}{\overline{\mathcal{M}}_{A,1}(G_{\mathbb{C}}/P)}
-\dim_{\mathbb{C}}(G_{\mathbb{C}}/P)+1\end{align*} So in conclusion,
\begin{align*}
\dim_{\mathbb{C}}{\Gamma_A(1\cdot P)^{\operatorname{op}}}&=
2\dim_{\mathbb{C}}{G_{\mathbb{C}}/P}-\dim_{\mathbb{C}}{\overline{\mathcal{M}}_{A,1}(G_{\mathbb{C}}/P)}
-1\\&=2\dim_{\mathbb{C}}{G_{\mathbb{C}}/P}-\dim_{\mathbb{C}}{\overline{\mathcal{M}}_{A,2}(G_{\mathbb{C}}/P)},
\end{align*}
and we are done

\item Type B and D Grassmannians: For a positive integer $m,$ we will write $m$ as $2n$ if $m$ is
even, and as $2n+1$ if $m$ is odd (here $n$ is a non-negative
integer number). Let $SO(m, \C)$ be the group of complex special
orthogonal matrices which preserves the standard nondegenerate
symmetric bilinear form defined on $\C^m.$ Let $\{e_1, \cdots,
e_n\}$ be the standard basis of $\R^n.$

The standard root system for the group $SO(2n+1, \C)$ is identified
with the set of vectors $R=\{\pm e_i,\, \pm(e_j\pm e_k):\, j\ne
k\}_{1\leq i, j \leq n}\subset \R^n$ with a choice of simple roots
given by $S=\{\alpha_1=e_1-e_2, \cdots, \alpha_{n-1}=e_{n-1}-e_n,
\alpha_n=e_n\},$ and Dynkin diagram
\begin{center}
\includegraphics{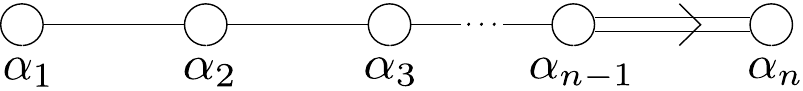}
\end{center}
The standard root system for the group $SO(2n, \C)$ is identified
with the set of vectors $R=\{\pm(e_j\pm e_k):  \ j\ne k\}_{1\leq i,
j \leq n}$ with simple roots given by $S=\{\alpha_1=e_1-e_2, \cdots,
\alpha_{n-1}=e_{n-1}-e_n, \alpha_n=e_{n-1}+e_n\}$ and with Dynkin
diagram
\begin{center}
\includegraphics{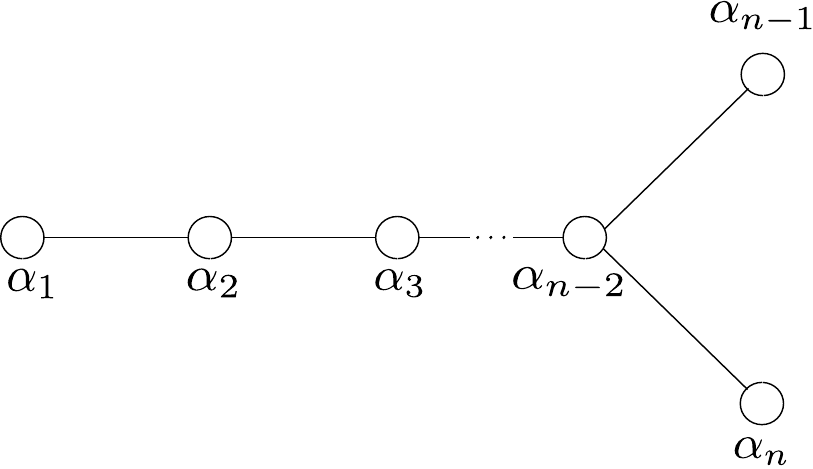}
\end{center}

Let $k\leq m/2$ be a positive integer. We denote by $OG(k, m)$ the
\textbf{Orthogonal Grassmannian manifold} of $k$-dimensional
isotropic subspaces in $\C^{m}$ with respect to the standard
nondegenerate symmetric bilinear form defined on $\C^m.$

When $k\ne m/2,$ the group $SO(m, \C)$ acts transitively on $OG(k,
m)$ and the orthogonal Grassmannannian $OG(k,m)$ is isomorphic to
the quotient $SO(m, \C)/P_{\alpha_k}.$ The two orthogonal
Grassmannians $OG(k, 2n)$ and $OG(k, 2n+1)$ are isomorphic.

When $k=m/2=n,$ the orthogonal Grassmannian $OG(n, 2n)$ is the union
of two $SO(2n, \C)$-orbits. These two $SO(2n, \C)$-orbits are
isomorphic to $SO(2n, \C)/P_{\alpha_{n-1}}$ and  $SO(2n,
\C)/P_{\alpha_n}.$ The two $SO(2n, \C)$-orbits of the orthogonal
Grassmannnian $OG(n, 2n)$ are isomorphic to the orthogonal
Grassmannian $OG(n, 2n+1).$

In summary, Grassmannian manifolds of type $B$ can be identified
with Grassmannian manifolds of type $D.$ The statement in this case
follows from the long root case.

\item Short root case (type G): Let $G=G_2$ and $T\subset G$ be the maximal torus whose Lie
algebra $\mathfrak{t}$ is identified with $\R^2$ and such that the
set
$$
S=\Bigl\{\alpha_1=\Bigl(-\frac{3}{2}, \frac{\sqrt{3}}{2}\Bigr),
\alpha_2=(1, 0)\Bigr\} \subset \mathfrak{t}^*\cong \R^2
$$
defines a set of simple root systems for $G$ with Dynkin diagram
\begin{center}
\includegraphics{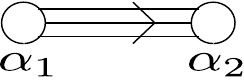}
\end{center}

Let us assume that $P\subset G_{\mathbb{C}}$ is the maximal
parabolic subgroup associated with the short simple root
$\alpha_2\in S.$ The homogeneous space $G_{\mathbb{C}}/P$ can be
considered as a homogenous space of $SO(7, \C):$ Note first that the
complex dimension of $G_{\mathbb{C}}/P$ is equal to $5.$ Let $w_1\in
\mathfrak{t}^*$ be the fundamental weight associated with
$\alpha_1.$ Let $L(w_1)=G_{\mathbb{C}}\times_P \C(w_1)$ be the line
bundle defined over $G_{\mathbb{C}}/P$ associated with the
fundamental weight $w_1.$ The irreducible representation
$H^0(G_{\mathbb{C}}/P,L(w_1))$ has dimension 7 (this computation can
be made by using for instance the Weyl dimensional formula). Thus,
$G_{\mathbb{C}}/P$ is embedded in the 6 dimensional projective space
$\mathbb{P}(H^0(G_{\mathbb{C}}/P,L(w_1)))\cong \mathbb{CP}^6.$ Under
this embedding, $G_{\mathbb{C}}/P$ is a $G_2$-homogenous
hypersurface and thus a nondegenerate quadric. A quadric in
$\mathbb{CP}^6$ is a complete homogeneous space for the special
orthogonal group $SO(7, \C).$ The result in this case follows from
the long root case for type $D$ Grassmannians.

\item Type C Grassmannians: Let $(\C^{2n}, \Omega)$ be the standard complex symplectic vector
space with complex coordinates $(z_1, \cdots, z_n, w_1, \cdots,
w_n)$ and with complex bilinear skew-symmetric form
$$
\Omega=\sum dz_i\wedge dw_i
$$
Let $Sp(n, \C)$ be the complex Lie group of linear transformation on
$\C^{2n}$ that preserves $\Omega.$ Let $\{e_1, \cdots, e_n\}$
denotes the standard basis of $\R^n.$

The standard root system of $Sp(n, \C)$ is identified with the set
$R=\{\pm e_i \pm e_j \ (i\ne j), \ \pm 2e_i\}_{1 \leq i, j\leq
n}\subset \R^n,$ with a choice of simple roots given by
$S=\{\alpha_1=e_1-e_2, \alpha_2=e_2-e_3, \cdots,
\alpha_{n-1}=e_{n-1}-e_n, \alpha_n=2e_n\}$ and Dynkin diagram
\begin{center}
\includegraphics{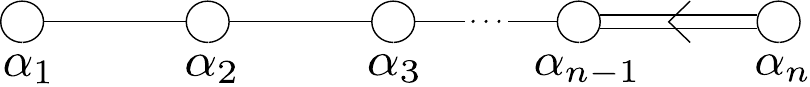}
\end{center}
For an integer $0<k\leq n,$ let $IG(k, 2n)$ denote the space of
$k$-dimensional isotropic subspaces of $\C^{2n},$ i.e,
$$
IG(k, 2n):=\{V^k\in G(k, 2n): \Omega|_{V^k}=0\}.
$$
When $k=n,$ the isotropic Grassmannian $IG(n, 2n)$ is the space of
Lagrangian subspaces of $\C^{2n}.$ The isotropic Grassmannian
manifold $IG(k, 2n)$ has dimension equal to $$2k(n-k) +
\dfrac{k(k+1)}{2}$$ and is isomorphic to the quotient of complex Lie
groups $Sp(2n, \C)/P_{\alpha_k},$ where $P_{\alpha_k}\subset Sp(n,
\C)$ denotes the maximal parabolic subgroup associated with the
simple root $\alpha_k\in S.$

Given a $J$-holomorphic curve $u:\cpi \to IG(k, 2n)$ of degree
$A:=\sigma_{P_{\alpha_k}}(s_{\alpha_k}),$ there exist linearly
independent vectors $v_0, v_1, v_2, \cdots, v_k \in \C^{2n}$ such
that
\begin{align*}
u:\cpi &\to IG(k, 2n)\\
[z_0, z_1] &\mapsto \operatorname{span}\{z_0v_0+z_0v_1, v_2, \cdots,
v_k\}
\end{align*}
In particular, for every $[z_0, z_1]\in \cpi,$ the vector subspace
$\operatorname{span}\{z_0v_0+z_0v_1, v_2, \cdots, v_k\}\subset
\C^{2n}$ is isotropic. We can associate a pair $(V^{k-1}, V^{k+1})$
of vector subspaces to the curve $u$ such that
$$
V^{k-1} \subset V^{k+1} \subset (V^{k-1})^{\Omega}\subset \C^{2n}
$$
Here, $V^{k-1}=\operatorname{span}\{v_2, \cdots, v_k\}$ and
$V^{k+1}=\operatorname{span}\{v_0, v_1, v_2, \cdots, v_k\}.$ The
pair of vector subspaces $(V^{k-1}, V^{k+1})$ determine uniquely, up
to reparametrization, the curve $u.$ This means that if $v:\cpi \to
IG(k, 2n)$ is another $J$-holomorphic of degree $A$ such that for
any $W^k\in v(\cpi)\subset IG(k, 2n)$
$$
V^{k-1} \subset  W^k \subset V^{k+1},
$$
then there exists $\varphi \in \operatorname{PSL}(2, \C)$ such that
$v\circ \varphi =u.$ This implies that the moduli space
$\overline{\mathcal{M}}_{A, 0}(IG(k, 2n), J)$ of unparameterized
$J$-holomorphic curves of degree $A$ in $IG(k, 2n)$ can be
identified with the set of pairs of subspaces
$$
\{ (V^{k-1}, V^{k+1}): V^{k-1} \subset V^{k+1} \subset
(V^{k-1})^{\Omega}\subset \C^{2n}\}
$$
Note that a pair of isotropic subspaces $V_1^k, V_2^k \in IG(k, 2n)$
are joined by a $J$-holomorphic curve of degree $A$ if
$$\dim_{\mathbb{C}}(V_1^k\cap V_2^k)=k-1 \ \ \text{and} \ \
\dim_{\mathbb{C}}(V_1^k+V_2^k)=k+1
$$
In particular, the degree $A$ neighborhood of the isotropic subspace
$\C^k\subset \C^{2n}$ is given by
$$
\Gamma_A(\C^k)=\{V^k\in IG(k, 2n):k-1\leq \dim(\C^k\cap V^k) \leq
\dim(\C^k+ V^k)\leq k+1\}
$$
We can compute the dimension of $\Gamma_A(\C^k)$ if we consider the
fibration
\begin{align*}
\Gamma_A(\C^k)-\{\C^k\} &\to  \{V^{k-1}: V^{k-1}\subset \C^k\}\cong G(k-1, k)\\
V^k &\mapsto \C^k\cap V^k
\end{align*}
The fiber of this fibration over any $(k-1)$-dimensional vector
subspace $V^{k-1}\subset \C^k$ is the set
$$
\{V^k\subset \C^{2n}: V^{k-1} \subset V^k \subset
(V^{k-1})^{\Omega}\}-\{\C^{k}\}
$$
of complex dimension $2n-2k+1.$ Thus,
$$
\dim_{\mathbb{C}}\Gamma_A(\C^k)=2n-k
$$
Likewise, the dimension of $\overline{\mathcal{M}}_{A, 0}(IG(k,
2n))$ can be computed by considering the fibration
\begin{align*}
\overline{\mathcal{M}}_{A, 0}(IG(k, 2n)) &\to IG(k-1, 2n)\\
(V^{k-1}, V^{k+1}) &\mapsto V^{k-1}
\end{align*}
This fibration has fiber isomorphic to $G(2, 2n-2k+2),$ so that
\begin{align*}
\dim_{\mathbb{C}}{\overline{\mathcal{M}}_{A, 0}(IG(k, 2n))}&=
\dim_{\mathbb{C}}{IG(k-1, 2n)}+\dim_{\mathbb{C}}{G(2,
2n-2k+2)}
\\&=\dim_{\mathbb{C}}{IG(k, 2n)}-2+2n-k,
\end{align*}
and we are done.

\item Short root case (Type F): Let $G$ be a compact Lie group of type $F_4$
and $T\subset G$ be the maximal torus whose Lie algebra is
identified with $\R^4$ and such that the set $S\subset
\mathfrak{t}^*\cong \R^4$ defined by
$$
\Bigl\{\alpha_1=(0, 1, -1, 0), \alpha_2=(0, 0, 1, -1), \alpha_3=(0,
0, 0, 1), \alpha_4=\Bigl(\dfrac{1}{2}, -\dfrac{1}{2}, -\dfrac{1}{2},
-\dfrac{1}{2}\Bigr) \Bigr\},$$ corresponds to the standard set of
simple roots of $G$ with Dynkin diagram

\begin{center}
\includegraphics{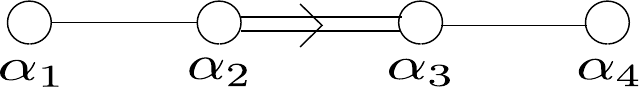}
\end{center}
Let $P\subset G_{\mathbb{C}}$ be the maximal parabolic associated to
a simple root $\alpha \in S$ and
$$
Z_A^{P}:=\{s_\beta \cdot P\in W/W_P: \beta \in R^+-R_P^+, \,\,
\check{\beta}=\check{\alpha}+ \Z \check{S}_P\}
$$
be the set of $T$-fixed points in $\Gamma_A(1\cdot P).$ According to
Lemma \ref{fixedpoints}, the set $Z_A^P$ contains a unique maximal
element $z_A^P$ with respect to the Bruhat order and
$$
\Gamma_A(1\cdot P)=X_P(z_A^P)
$$
Now we check that for the parabolic subgroup $P$ associated with
either the short simple root $\alpha_3$ or $\alpha_4$ we have that
$$
c_1(A)=\dim_{\mathbb{C}}(\Gamma_A(1\cdot P))+1=l(z_A^P)+1
$$
\begin{enumerate}
\item Let $P$ be the maximal parabolic subgroup associated with $\alpha=\alpha_4.$
The following figure shows the minimum length representatives of
cosets in $Z_A^P$ ordered with respect to the Bruhat order:
\begin{center}
\includegraphics{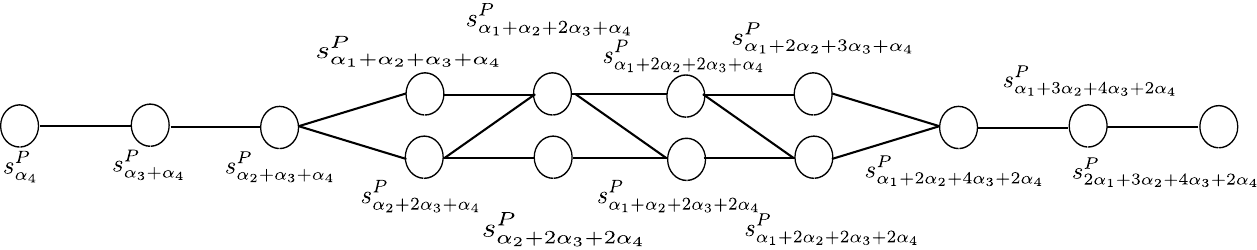}
\end{center}
The maximal element of $Z_A^P$ is the coset
$z_A^P:=s_{2\alpha_1+3\alpha_2+4\alpha_3+2\alpha_4}\cdot P.$ The
minimum length representative of $z_A^P$ in $W^P$ is
$$
s_{\alpha_1}s_{\alpha_2}s_{\alpha_3}s_{\alpha_4}s_{\alpha_2}s_{\alpha_3}s_{\alpha_1}s_{\alpha_2}s_{\alpha_3}s_{\alpha_4}
$$
with length equal to $10.$ Thus $\dim_{\mathbb{C}}\Gamma_A(1\cdot
P)=10.$ Using Proposition \ref{chern} we get
\begin{align*}
c_1(A)&=\langle c_1(T(G_{\mathbb{C}}/P)), \check{\alpha}_4
\rangle=\langle 11(\alpha_1+2\alpha_2+3\alpha_3+2\alpha_4),
\check{\alpha}_4 \rangle\\&=11(1\cdot 0+ 2\cdot 0 +3\cdot(-1)+2\cdot
2)=11\\&=\dim_{\mathbb{C}}\Gamma_A(1\cdot P)+1
\end{align*}

\item Let $P$ Be the maximal parabolic subgroup associated with $\alpha=\alpha_3.$
The Hasse diagram of minimum length representatives of cosets in
$Z_A^P$ ordered with respect to the Bruhat order is shown below
\begin{center}
\includegraphics[scale=0.9]{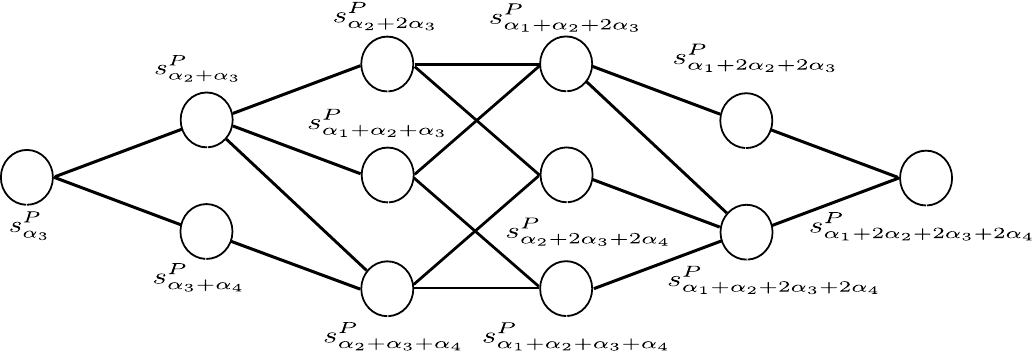}
\end{center}
The maximal element of $Z_A^P$ is
$z_A^P:=s_{\alpha_1+2\alpha_2+2\alpha_3+2\alpha_4}\cdot P.$ The
minimum length representative  of $z_A^P$ in $W^P$ is
$$
s_{\alpha_4}s_{\alpha_2}s_{\alpha_3}s_{\alpha_1}s_{\alpha_2}s_{\alpha_3}
$$
with length equal to $6.$ Thus $\dim_{\mathbb{C}}\Gamma_A(1\cdot
P)=6.$ Using Proposition \ref{chern} we get
\begin{align*} c_1(A)&=\langle
c_1(T(G_{\mathbb{C}}/P)), \check{\alpha}_3 \rangle=\langle
7(2\alpha_1+4\alpha_2+6\alpha_3+3\alpha_4),
\check{\alpha}_3\rangle\\&=7(2\cdot 0+ 4\cdot(-2)+ 6\cdot 2
+3\cdot(-1))=7\\&=\dim_{\mathbb{C}}(\Gamma_A(1\cdot P))+1,
\end{align*}
and we are done

\end{enumerate}
\end{itemize}
\end{proof}

Now we are ready to state our upper bound for the Gromov width of
Grassmannian manifolds.

\begin{Theorem}\label{alexcastro3}
Let $G$ be a compact connected simple Lie group with Lie algebra
$\mathfrak{g}.$ Let $T\subset G$ be a maximal torus with a choice of
simple roots $S\subset \mathfrak{t}^*.$ For $\lambda \in
\mathfrak{t}^* \subset \mathfrak{g}^*,$ let $\mathcal{O}_\lambda$ be
the coadjoint orbit passing through $\lambda$  and $\w_\lambda$ be
the Kostant-Kirillov-Souriau form defined on $\mathcal{O}_\lambda.$
Assume that there is a maximal parabolic subgroup $P\subset
G_{\mathbb{C}}$ associated  with a simple root $\alpha\in S$ such
that $\mathcal{O}_\lambda\cong G_{\mathbb{C}}/P,$ then
$$
\operatorname{Gwidth}(\mathcal{O}_\lambda, \w_\lambda)\leq |\langle
\lambda\, , \check{\alpha} \rangle|
$$
\end{Theorem}
\begin{proof}
The symplectic manifold $(\mathcal{O}_\lambda, \w_\lambda)$ is
semipositive because when $P\subset G_{\mathbb{C}}$ is a maximal
parabolic subgroup, the homology group $H_2(G/P; \Z)$ is cyclic. Let
$J$ be the complex structure defined on $(\mathcal{O}_\lambda,
\w_\lambda)$ coming from the quotient of complex Lie groups
$G_{\mathbb{C}}/P.$ The complex structure $J$ is regular (see e.g.
McDuff and Salamon \cite{mcduff}[Proposition 7.4.3]). Let $A$ be the
Schubert class that cyclic generates $H_2(G_{\mathbb{C}}/P; \Z)$ and
$\Gamma_A(1\cdot P)$ be the degree $A$ neighborhood of $1\cdot P.$
The previous Theorem states that
$$
\operatorname{GW}_{A,2}^{\,\operatorname{alg}}\bigl(\operatorname{PD}[1\cdot
P],\operatorname{PD}[\Gamma_A(1\cdot P)^{\operatorname{op}}]\bigr)=
\operatorname{GW}_{A,2}^J\bigl(\operatorname{PD}[1\cdot
P],\operatorname{PD}[\Gamma_A(1\cdot P)^{\operatorname{op}}]\bigr)=1
$$
Thus, by Lemma \ref{nsq2}
$$
\operatorname{Gwidth}(\mathcal{O}_\lambda, \w_\lambda)\leq
\w_\lambda(A)=|\langle \lambda\, , \check{\alpha} \rangle| \qedhere
$$
\end{proof}

\section{Upper bound for the Gromov width of coadjoint orbits of
compact Lie groups}

The problem of finding upper bounds for the Gromov width of
coadjoint orbits of compact Lie groups has already been addressed by
Masrour Zoghi in his Ph.D thesis \cite{zoghi} where he has
considered the problem of determining the Gromov width of
\textit{regular} coadjoint orbits of compact Lie groups. In this
section, we extend Zoghi's theorem to coadjoint orbits of compact
Lie groups that are not necessarily regular. We estimate from above
the Gromov width of arbitrary coadjoint orbits of compact Lie group
by computing Gromov-Witten invariants on holomorphic fibrations
whose fibers are isomorphic to Grassmannian manifolds.

Let $G$ be a compact Lie group, $\mathfrak{g}$ be its Lie algebra
and $\mathfrak{g}^*$ be the dual of this Lie algebra. Let
$\lambda\in \mathfrak{g}^*$ and $\mathcal{O}_\lambda\subset
\mathfrak{g}^*$ be the coadjoint orbit passing through $\lambda.$
Let us assume that $\mathcal{O}_\lambda \cong G_{\mathbb{C}}/P,$
where $P\subset G_{\mathbb{C}}$ is a parabolic subgroup of
$G_{\mathbb{C}}.$ We endow $G_{\mathbb{C}}/P$ with a K\"ahler
structure coming from its identification with $\mathcal{O}_\lambda.$
This K\"ahler structure and the one defined on $\mathcal{O}_\lambda$
would be denoted indistinguishably by $(\w_\lambda, J).$

Let $T\subset G$ be a maximal torus and let $B\subset
G_{\mathbb{C}}$ be a Borel subgroup with $T_{\mathbb{C}}\subset
B\subset P.$ Let $W=N(T)/T$ be the associated Weyl group. Let $R$ be
the corresponding set of roots and $S$ be the corresponding system
of simple roots. Let $W_P$ be the Weyl group of $P$ and $S_P$ be the
subset of simple roots whose corresponding reflections are in $W_P.$

The second homology group $H_2(G_{\mathbb{C}}/P; \Z)$ is freely
generated as a $\Z$-module by the set of Schubert classes
$\{[X_P(s_\alpha)]\}_{\alpha\in S-S_P}.$ We denote the Schubert
class $[X_P(s_\alpha)]$ by $A_\alpha.$ The symplectic area of
$A_\alpha$ is equal to $\w_\lambda(A_\alpha)=|\langle \lambda\,,
\check{\alpha} \rangle|.$

Now we show that for every class $A_\alpha$ there is a non-vanishing
Gromov-Witten invariant with one its constrains being Poincar\'e
dual to the class of a point.

\begin{Theorem}\label{alexcastro4}
Let $P\subset G_{\mathbb{C}}$ be a parabolic subgroup, $A_\alpha\in
H_2( G_{\mathbb{C}}/P, \Z)$ be the class associated with a simple
root $\alpha\in S-S_P$ and $\Gamma_{A_\alpha}(1\cdot P)$ be the
degree $A_\alpha$ neighborhood of $1\cdot P.$ Then
$$
\operatorname{GW}_{A_{\alpha},2}^{\,\operatorname{alg}}\bigl(\operatorname{PD}[1\cdot
P],\operatorname{PD}[\Gamma_{A_\alpha}(1\cdot
P)^{\operatorname{op}}]\bigr)=1
$$
\end{Theorem}
\begin{proof}
We show that
$$
c_1(T(G_{\mathbb{C}}/P))(A_\alpha)=1+\dim_{\mathbb{C}}(\Gamma_{A_\alpha}(1\cdot
P)),
$$
and the result will follow from Theorem \ref{alexcastro}.

For the simple root $\alpha \in S-S_P,$ we have a parabolic subgroup
$Q\subset P$ with $S_Q=S_P\sqcup \{\alpha\}$ and a fibration
$$
\pi_\alpha:G_{\mathbb{C}}/P \to G_{\mathbb{C}}/Q.
$$
The fibration $\pi_\alpha$ is holomorphic with respect to the
complex structures defined on the quotients of complex Lie groups
$G_{\mathbb{C}}/P, G_{\mathbb{C}}/Q.$

The fiber $Q/P$ can be identified with the quotient of a simple
complex Lie group and a maximal parabolic subgroup. We also have an
inclusion map
$$
\iota_\alpha: Q/P  \hookrightarrow G_{\mathbb{C}}/P.
$$
Let $A$ be the Schubert class that cyclic generates $H_2(Q/P; \Z).$
Note that $\iota_{\alpha*}(A)=A_\alpha \in H_2(G_{\mathbb{C}}/P;\,
\Z)$ and $\pi_{\alpha*}(A_\alpha)=0\in H_2(G_{\mathbb{C}}/Q; \,
\Z).$

We have an exact sequence of vector bundles over $Q/P$
$$
0\to T(Q/P) \xrightarrow{d\iota_\alpha} T(G_{\mathbb{C}}/P)|_{Q/P}
\xrightarrow{d\pi_\alpha} Q/P\times
\mathfrak{g}_{\mathbb{C}}/\mathfrak{q} \to 0
$$
Thus,
$$
c_1(T(G_{\mathbb{C}}/P)|_{Q/P})=c_1(T(Q/P))+c_1(
Q/P\times\mathfrak{g}_{\mathbb{C}}/\mathfrak{q})=c_1(T(Q/P))
$$
The projection formula of Chern classes implies that
\begin{align*}
c_1(T(G_{\mathbb{C}}/P)|_{Q/P})(A)&=c_1(\iota_\alpha^*(T(G_{\mathbb{C}}/P)))(A)\\&=c_1(T(G_{\mathbb{C}}/P))(\iota_{\alpha*}(A))\\&=
c_1(T(G_{\mathbb{C}}/P))(A_\alpha)
\end{align*}
and we get
$$
c_1(T(Q/P))(A)=c_1(T(G_{\mathbb{C}}/P))(A_\alpha)
$$
Now we describe the degree $A_\alpha$ neighborhood
$\Gamma_{A_{\alpha}}(1\cdot P)$ of $1\cdot P\in G_{\mathbb{C}}/P.$
Let $u:\cpi \to G_{\mathbb{C}}/P$ be a $J$-holomorphic map of degree
$A_\alpha.$ The map
$$
\pi_\alpha\circ u:\cpi \to G_{\mathbb{C}}/Q
$$
is holomorphic and its degree is equal to $(\pi_\alpha\circ
u)_*[\cpi]=(\pi_\alpha)_*[A_\alpha]=0\in H_2(G_{\mathbb{C}}/Q; \Z).$
Therefore, the map $\pi\circ u:\cpi \to G_{\mathbb{C}}/Q$ is
constant and the image of $u:\cpi \to G_{\mathbb{C}}/P$ is totally
contained in a fiber of $\pi_\alpha:G_{\mathbb{C}}/P \to
G_{\mathbb{C}}/Q.$ If the curve $u$ passes through $1\cdot P\in
G_{\mathbb{C}}/P,$ we can identify $u$ with a curve of degree $A$ in
the fiber $Q/P.$ As a consequence, the degree $A_\alpha$
neighborhood $\Gamma_{A_\alpha}(1\cdot P)\subset G_{\mathbb{C}}/P$
can be identified with the degree $A$ neighborhood $\Gamma_A(1\cdot
P)\subset Q/P,$ and in particular they share the same dimension. By
the proof of Theorem \ref{alexcastro2},
$$
c_1(T(Q/P))(A)=1+\dim_{\mathbb{C}}(\Gamma_A(1\cdot P)),
$$
and thus
$$
c_1(T(G_{\mathbb{C}}/P))(A_\alpha)=1+\dim_{\mathbb{C}}(\Gamma_{A_\alpha}(1\cdot
P)),
$$
and we are done.
\end{proof}

The result that follows is the main theorem of this paper and gives
an upper bound for the Gromov width of coadjoint orbits of compact
Lie groups that are not necessarily regular.
\begin{Theorem}
Let $G$ be a compact connected simple Lie group with Lie algebra
$\mathfrak{g}.$ Let $T\subset G$ be a maximal torus and let
$\check{R}\subset \mathfrak{t}$ be the corresponding system of
coroots. We identify the dual Lie algebra $\mathfrak{t}^*$ with the
fixed points of the coadjoint action of $T$ on $\mathfrak{g}^*.$ Let
$\lambda \in \mathfrak{t}^* \subset \mathfrak{g}^*,$
$\mathcal{O}_\lambda$ be the coadjoint orbit passing through
$\lambda$  and $\w_\lambda$ be the Kostant--Kirillov--Souriau form
defined on $\mathcal{O}_\lambda,$ then
\[
\operatorname{Gwidth}(\mathcal{O}_\lambda, \w_\lambda) \leq
\min_{\substack{\check{\alpha} \in \check{R} \\ \langle \lambda,
\check{\alpha} \rangle\ne 0}} |\langle \lambda, \check{\alpha}
\rangle|
\]
\end{Theorem}
\begin{proof}
Let $P\subset G_{\mathbb{C}}$ be a parabolic subgroup such that
$\mathcal{O}_\lambda \cong G_{\mathbb{C}}/P.$ We establish the
following convention:  for every $\tilde{\lambda}\in \mathfrak{t}^*$
such that the coadjoint orbit $\mathcal{O}_{\tilde{\lambda}}$ is
isomorphic with $G_{\mathbb{C}}/P,$ we are going to see the
Kostant-Kirillov-Souriau form $\w_{\tilde{\lambda}}$ defined on
$\mathcal{O}_{\tilde{\lambda}}$ as a symplectic form on
$G_{\mathbb{C}}/P.$

Let $\{\lambda_n\}_{n\in \mathbb{Z}_{> 0}}\subset \mathfrak{t}^*$ be
a sequence that converges to $\lambda.$ Assume that for every $n\in
\Z_{>0},$ the coadjoint orbit $\mathcal{O}_{\lambda_n}$ is
isomorphic with $G_{\mathbb{C}}/P$ and the Kostant-Kirillov-Souriau
form $\w_{\lambda_n}$ represents a cohomology class
$[\w_{\lambda_n}]$ with rational coefficients.

Let $\tilde{J}\in \mathcal{J}(G_{\mathbb{C}}/P, \w_{\lambda}).$ Let
$\tilde{J}_n\in \mathcal{J}(G_{\mathbb{C}}/P, \w_{\lambda_n})$ be a
sequence that converges to $\tilde{J}$ in the $C^\infty$-topology.

For every $n\in \Z_{>0},$ the symplectic manifold
$(G_{\mathbb{C}}/P, \w_{\lambda_n})$ meets all the requirements of
Theorem \ref{nsq4}: the almost complex structure $J$ coming from the
quotient of complex Lie groups $G_{\mathbb{C}}/P$ is regular and
compatible with the Kostant-Kirillov-Souriau form $\w_{\lambda_n}.$
For any simple root $\alpha \in S-S_P,$ the homology class
$A_\alpha\in H_2(G_{\mathbb{C}}/P; \Z)$ is $J$-indecomposable and
according to the previous theorem there exists a cohomology class
$a\in H^*(G_{\mathbb{C}}/P; \Z)$ such that
$$
\operatorname{GW}_{A_{\alpha},2}^{J}\bigl(\operatorname{PD}[1\cdot
P],\, a \bigr)\ne 0
$$
Thus by Theorem \ref{nsq4}, for every point $p\in G_{\mathbb{C}}/P,$
we can find a $\tilde{J}_n$-holomorphic sphere $u_n:\cpi \to
G_{\mathbb{C}}/P$ of degree $B_n $ passing through $p$ with
$0<\w_{\lambda_n}(B_n)\leq \w_{\lambda_n}(A_\alpha).$ The Gromov
compactness theorem implies that there exists a
$\tilde{J}$-holomorphic curve of degree $B$ passing through $p$ with
$0<\w_\lambda(B) \leq \w_\lambda(A_\alpha).$ By Theorem \ref{nsq},
$$
\operatorname{Gwidth}(\mathcal{O}_\lambda, \w_{\lambda}) \leq
\w_\lambda(A_\alpha)=|\langle \lambda, \check{\alpha} \rangle|
$$
The above inequality  holds for any $\alpha\in S-S_P,$ and as
consequence for any $\alpha \in R^+-R^+_P,$ and we are done.

\end{proof}

\section{Appendix: Fibrations}

Let $G$ be a compact simple Lie group. Let $T\subset G$ be a maximal
torus, $B\subset G_{\mathbb{C}}$ be a Borel subgroup with
$T_{\mathbb{C}}\subset B$ and $S$ be the corresponding system of
simple roots. Let $W=N(T)/T$ be the associated Weyl group. For a
parabolic subgroup $P \subset G_{\mathbb{C}}, B\subset P,$ let
$W_P=N_P(T)/T$ be the Weyl group of $P$ and $S_P$ be the subset of
simple roots of $S$ whose corresponding reflections are in $W_P.$
Let $W^P\subset W$ be the set of all minimum length representatives
for cosets in $W/W_P.$ For $w\in W^P,$ let $X_P(w)\subset
G_{\mathbb{C}}/P$ be the Schubert variety associated with $w\in
W^P.$

For a pair of parabolic subgroups $P,  Q \subset G_{\mathbb{C}},$
such that  $B\subset P\subset Q, $ we have a quotient map
$G_{\mathbb{C}}/P \to G_{\mathbb{C}}/Q.$ We want to study the images
and preimages of Schubert varieties under these quotient maps.

\begin{lemma}[\textbf{Stumbo \cite{stumbo}}]
For parabolic subgroups $P, Q \subset G_{\mathbb{C}}$ such that
$B\subset P\subset Q$ define
\begin{align*}
W_Q^{P}&:=\{ w \in W_{Q}: l(ws) > l(w) \text{ for $s\in
S_P$}\}\\&=\text{ minimum length representatives of cosets in
$W_{Q}/W_P$}
\end{align*}
Given $w\in W^P,$ there is a unique $w^Q\in W^Q$ and a unique
$w_Q^P\in W_Q^P$ such that $w=w^Qw^P_Q.$ Their lengths satisfy
$l(w)=l(w^Q)+l(w^P_Q).$
\end{lemma}

\begin{lemma}\label{longest}
For parabolic subgroups $P, Q \subset G_{\mathbb{C}}$ such that
$B\subset P\subset Q,$ let $w_q^p, w_p$ and $w_q$ be the longest
elements in $W_Q^P,  W_P$ and $W_Q,$ respectively. Then,
$w_q^p=w_qw_p.$
\end{lemma}
\begin{proof}
Let $w_0$ be the longest element in $W.$ The quotient map $\pi : W
\to W^P\cong W/W_P$ is order preserving and thus the longest element
in $W^P$ is $\pi(w_0)$. By the previous lemma $w_0 = \pi(w_0)w_p$,
so that $\pi(w_0) = w_0w^{-1}_p.$ Similarly, for the quotient map
$\pi' : W \to W^Q\cong W/W_Q,$ we have that $\pi'(w_0) =
w_0w^{-1}_q$ is the longest element in $W^Q.$ Using again the
previous Lemma, we have that the permutation $\pi'(w_0)w_q^p$ is the
longest permutation in $W^P$ . So that $\pi'(w_0)w_q^p = \pi(w_0)$,
and thus $w_0w^{-1}_q w_q^p = w_0w^{-1}_p$ or $w_q^p = w_qw^{-1}_p.$
The longest element in any finite Coxeter group is idempotent. Thus
$w^2_p= e,$ and we are done.
\end{proof}

\begin{Proposition}\label{proposition}
For parabolic subgroups $P, Q\subset G_{\mathbb{C}}$ such that
$P\subset Q \subset G_{\mathbb{C}},$ let
$$
\pi_q:G_{\mathbb{C}}/P \to G_{\mathbb{C}}/Q
$$
be the quotient fibration. If we decompose $w\in W^P$ as $w^Qw_Q^P,$
where $w^Q\in W^Q$ and $w_Q^P\in W_Q^P,$ then
$\pi_{q*}(X_P(w))=X_Q(w^Q).$ On the other hand, if $\tilde{w} \in
W^Q,$ then $\pi_q^*(X_Q(\tilde{w}))=X_P(\tilde{w}w_q^p),$ where
$w_q^p$ is the longest element in $W_Q^P.$
\end{Proposition}
\begin{proof}
The map $\pi_q:G_{\mathbb{C}}/P \to G_{\mathbb{C}}/Q$ is
$B$-equivariant and closed (this is a consequence of for example the
closed map lemma). This implies that Schubert cells, which are
$B$-orbits, and Schubert varieties, which are their closures, in
$G_{\mathbb{C}}/P$ are mapped to Schubert cells and Schubert
varieties in $G_{\mathbb{C}}/Q,$ respectively.

For $w\in W^P\subset W,$ there exist unique $w^Q\in W^Q$ and
$w_{Q}^P\in W_Q^P\subset W_Q$ such that $w=w^Qw_{Q}^P$ and
$l(w)=l(w^Q)+l(w_{Q}^P).$ The Schubert cell $C_P(w)=BwP/P \subset
G_{\mathbb{C}}/P$ is mapped to the Schubert cell
$C_Q(w^Q)=Bw^QQ/Q\subset G_{\mathbb{C}}/Q$ via $\pi,$ and
$$
\pi_{q*}(X_P(w))=\pi_{q*}(\overline{C_P(w)})=\overline{\pi_{q*}(C_P(w))}=\overline{C_Q(w^Q)}=X_Q(w^Q).
$$
On the other hand, if $\tilde{w}\in W^Q,$ then
$$
\pi_q^{*}(C_Q(\tilde{w}))=\bigsqcup_{\substack{\quad\,v\in W^P\\v^Q=
\tilde{w}}} C_P(v).
$$
The maximum element, with respect to the Bruhat order defined on
$W^P,$ in the set $\{v\in W^P: v^Q= \tilde{w}\}$ is $\tilde{w}
w_q^p,$ where $w_q^p$ denotes the longest element in $W_Q^P.$ Since
$\pi$ is a continuous map, we have that
$$
\pi_q^{*}(X_Q(\tilde{w}))=\pi_q^{*}(\overline{C_Q(\tilde{w})})=\overline{\bigsqcup_{\substack{\quad\,v\in
W^P \\ v^Q= \tilde{w}}} C_P(v)}=\bigsqcup_{\substack{\quad\,v\in W^P \\
v \leq_B \tilde{w} w_q^p}} C_P(v)=X_P(\tilde{w} w_q^p).
$$
\end{proof}
The following two technical lemmas are needed it in the proof of
Theorem \ref{alexcastro3}:
\begin{lemma}\label{mainlemma}
Let $\alpha\in S$ be a simple root and $N(\alpha)\subset S$ be the
neighbors of $\alpha$ in the Dynkin diagram of $G,$ i.e., the simple
roots connected to $\alpha$ by an edge in the Dynkin diagram of $G.$
Let $P, R \subset G_{\mathbb{C}} $ be the parabolic subgroups such
that $S_P=S-\{\alpha\}, S_{R}=S-(N(\alpha) \cup \{\alpha\}),$
respectively. Then
$$
W_{P}^{R}\cdot s_\alpha \subset W^{P}
$$
\end{lemma}
\begin{proof}
Let $w\in W_{P}^{R}.$ We write $w=s_1\cdot \ldots \cdot s_r$ where
$s_1, \ldots, s_r$ are simple reflections in $S_P.$ Suppose that
there exists a simple reflection $t$ in $S_P$ such that $l(ws_\alpha
t)<l(ws_\alpha).$ By the Exchange Principle (see e.g. Humphreys
\cite{humphreys}),
$$
ws_\alpha t=s_1\cdot \ldots \cdot \hat{s_i} \cdot \ldots \cdot
s_rs_\alpha$$ for some $i,$ in particular $s_\alpha t s_\alpha \in
W_P.$ We now consider two cases and see that this is not possible:
\begin{enumerate}
\item Suppose that $s_\alpha t = t s_\alpha.$ Thus $t\notin
N(\alpha),$ but $t\ne s_\alpha$ so $t\in S-(N(\alpha) \cup
\{s_\alpha\})=S_{R}.$ As $w\in W_P^{R}$
$$
l(wt)>l(w),
$$
hence
$$
l(ws_{\alpha}t)=l(wts_{\alpha})=l(wt)+1>l(w)+1=l(ws_\alpha),
$$
which contradicts our asumption of having $l(ws_\alpha
t)<l(ws_\alpha)$.

\item Suppose that $s_\alpha t \ne ts_\alpha.$
If $l(s_\alpha t s_\alpha)\ne 3,$ by the Deletion Principle (see
e.g. Humphreys \cite{humphreys}) either $s_\alpha t
s_\alpha=s_\alpha,$ or $s_\alpha t s_\alpha=t,$ which are not
possible. So $l(s_\alpha t s_\alpha)=3.$ Now, clearly $l(s_\alpha
t)=l(s_\alpha t s_\alpha s_\alpha)=2< l(s_\alpha t s_\alpha),$ so if
$s_\alpha t s_\alpha =s_1s_2s_3,$ for some simple reflections $s_1,
s_2, s_3 \in S_P,$ by the Exchange Principle $s_\alpha t \in W_P$
which would imply that $s_\alpha\in W_P,$ a contradiction.
\end{enumerate}
\end{proof}

\begin{lemma}\label{echeck}
Let $\alpha \in S$ be a simple root, $P \subset G_{\mathbb{C}}$ be
the maximal parabolic subgroup associated with $\alpha$ and $A$ be
the Schubert class that cyclic generates $H_2( G_{\mathbb{C}}/P;
\Z).$ Let $P, Q, R\subset G_{\mathbb{C}}$ be the parabolic subgroups
with $S_P=S-\{\alpha\}, S_{Q}=S-N(\alpha)$ and
$S_{R}=S-(N(\alpha)\cup \{\alpha\}),$ where $N(\alpha)\subset S$
denotes the neighbors of $\alpha$ in the Dynkin diagram of $G.$ Let
$\pi_p:G_{\mathbb{C}}/R \to G_{\mathbb{C}}/P$ and
$\pi_q:G_{\mathbb{C}}/R\to G_{\mathbb{C}}/Q$ be the corresponding
quotient maps. Then
$$
\pi_{q*}(\pi_p^{*}(1\cdot P))=X_{Q}(w_{p}w_{r})
$$
where $w_p$ and $w_{r}$ are the longest elements in $W_P$ and
$W_{R},$ respectively. In addition,
$$
\pi_{p*}(\pi_q^{*}(X_{Q}(w_pw_r)))=X_P(w_pw_{r}s_\alpha),
$$
where $s_\alpha$ is the simple reflection associated to $\alpha\in
S.$
\end{lemma}
\begin{proof}
By the Proposition \ref{proposition}, we have that
$\pi_p^{*}(X_P(e))=X_R(w_{p}^{r})=X_{R}(w_{p}w_{r}),$ so
$$
\pi_{q*}(\pi_p^{*}(1\cdot P)=\pi_{q*}(X_{R}(w_{p}^{r})).
$$
Note that $W_P^{R}\subset W^{Q}.$ In particular $w_{p}^{r}\in
W^{Q},$ and thus
$$
\pi_{q*}(\pi_p^{*}(1\cdot P))=X_{Q}(w_{p}^{r})
$$
Finally, Lemma \ref{mainlemma} implies that $w_p^rs_\alpha\in W^P,$
thus
$$
\pi_{p*}(\pi_q^{*}(X_{Q}(w_p^r)))=\pi_{p*}(X_{R}(w_p^rs_\alpha))=X_P(w_p^rs_\alpha),
$$
and we are done.
\end{proof}

\renewcommand{\refname}{Bibliography}
\bibliographystyle{abbrv}
\bibliography{biblography}
\nocite{*}

\end{document}